\documentclass[11pt,reqno,letterpaper]{amsart} 

\usepackage[utf8]{inputenc}    
\usepackage[T1]{fontenc}       

\usepackage{amsmath}           
\usepackage{amssymb}           
\usepackage{amsthm}            
\usepackage{mathtools}         
\usepackage{mathrsfs}          
\usepackage{bbm}               
\usepackage{bm}                
\usepackage{dsfont}            

\usepackage{geometry}          

\usepackage{graphicx}          
\usepackage{xcolor}            
\usepackage{booktabs}          
\usepackage{enumitem}          
\usepackage[normalem]{ulem}    

\usepackage[colorlinks=true, linkcolor=blue, citecolor=blue, urlcolor=black]{hyperref}

\numberwithin{equation}{section} 

\theoremstyle{plain} 
\newtheorem{thm}{Theorem}[section]
\newtheorem{lem}[thm]{Lemma}
\newtheorem{prop}[thm]{Proposition}

\newtheorem{conj}[thm]{Conjecture}
\newtheorem{assu}[thm]{Assumption}

\theoremstyle{definition} 
\newtheorem{defi}[thm]{Definition}

\theoremstyle{remark} 
\newtheorem{rmk}[thm]{Remark}

\newcommand{\N}{\mathbb{N}}

\newcommand{\R}{\mathbb{R}}

\newcommand{\E}{\mathbb{E}}   
\newcommand{\Pb}{\mathbb{P}}  
\newcommand{\Law}{\mathrm{Law}} 

\newcommand{\rd}{\mathrm{d}}  
\newcommand{\defeq}{\coloneqq} 

\DeclareMathOperator{\supp}{supp}

\def\XXint#1#2#3{{\setbox0=\hbox{$#1{#2#3}{\int}$ }
\vcenter{\hbox{$#2#3$ }}\kern-.6\wd0}}

\title[Propagation of dissociatedness and graphon sampling lemma]{Non-exchangeable mean-field theory for adaptive weights: propagation of dissociatedness and graphon sampling lemma}
\author{Datong Zhou}
\address{Laboratoire Jacques-Louis Lions and Laboratoire de Probabilités, Statistique et Modélisation, Sorbonne Université, 75005 Paris, France}
\email{datong.zhou@sorbonne-universite.fr}
\thanks{This project has received funding from the European Union’s Horizon 2020 research and innovation programme under the Marie Skłodowska-Curie grant agreement No. 101034255.} 

\begin{document}

\maketitle

\begin{abstract}

We develop a mean-field theory for large, non-exchangeable particle (agent) systems where the states and interaction weights co-evolve in a coupled system of SDEs. A first main result is the establishment of the propagation of dissociatedness, a conceptual generalization of the classical propagation of chaos that accommodates the intrinsic local correlations between particles and their weights. The limiting McKean-Vlasov process is characterized by an Aldous-Hoover representation on a filtered probability space, beyond the standard one-particle law (or a family thereof). Paralleling the classical equivalence between propagation of chaos and the convergence of empirical measures to the one-particle law, we show that the propagation of dissociatedness corresponds to the convergence of the empirical structure under a distance unifying the Wasserstein distance for particles and the cut distance for weights. This quantitative stability is grounded in an adaptation of the sampling lemma from dense graph theory, analogous to the classical concentration results for empirical measures in the Wasserstein distance.

\end{abstract}

\tableofcontents

\section{Overview} \label{sec:overview}

\subsection{The paradigm of mean-field limits: from exchangeable to non-exchangeable, and plasticity}

The modern paradigm for studying large systems of interacting particles was established over 150 years ago by Boltzmann \cite{boltzmann1872weitere}, marking a revolutionary shift from tracking individual microscopic trajectories to adopting a statistical description. 
The challenge of rigorously deriving macroscopic laws from microscopic dynamics, famously posed by Hilbert as his sixth problem, has since driven the development of kinetic theory. 
Foundational works by Kac \cite{kac1956foundations}, who introduced the notion of ``propagation of chaos'', and by Vlasov \cite{vlasov1950many} and McKean \cite{mckean1969propagation}, solidified the \emph{mean-field limit} as one of the central concepts in probability theory and statistical physics.

In the 21st century, the applications of mean-field theory have expanded dramatically beyond its origins in statistical physics.
It is now an indispensable tool in biomathematics for modeling collective behaviors like flocking \cite{degond2018mathematical, vicsek2012collective} and neuron spiking models \cite{caceres2011analysis, delarue2015global, fournier2016toy}, in social sciences to understand opinion dynamics \cite{degond2017continuum, naldi2010mathematical}, as well as in data science for analyzing the training of artificial neural networks \cite{chizat2018global,mei2018mean}. 
The remarkable success of this theory historically hinges on a core symmetry assumption often called \emph{exchangeability}. 
Mathematically, this implies that the interaction mechanism is indifferent to the permutation of particles, allowing the large-scale behavior of the system to be described by a single representative distribution.

However, many critical models at the frontier of modern science are defined by the breaking of such symmetry. 
In complex networks, such as the brain's connectome \cite{hulse2021connectome} or social graphs, interactions are highly structured and \emph{non-exchangeable}\footnote{We adopt the term ``non-exchangeable'' following recent literature in kinetic theory \cite{ayi2024large, ayi2024mean, jabin2025mean, jabin2023mean}, highlighting the contrast with classical mean-field theory, where particle exchangeability allows the limit to be fully described by a one-particle distribution; here, the heterogeneity of interaction weights renders such a straightforward description insufficient. 
However, it is important to note that at a higher level, any multi-agent system retains a fundamental symmetry: the dynamics is invariant under simultaneous permutation of the labels (indices), provided that all determining factors (here, states and weights) are treated as part of the system configuration rather than a fixed background.
From this perspective, our study is consistent with the theory of \emph{exchangeable arrays} in the probabilistic literature \cite{aldous1981representations, hoover1979relations, kallenberg1989representation}.}. 
Recent mathematical progress has begun to address this by incorporating the concept of \emph{graph limits} (e.g., graphons) into the mean-field framework \cite{bayraktar2023graphon, chiba2018mean, jabin2025mean, kaliuzhnyi2018mean}. 
In these static non-exchangeable settings, which particle can interact with which depends on a prescribed, possibly dense, weighted graph. 
The continuum limit is described by an extended Vlasov equation coupled with a static graphon kernel, which is typically fixed \emph{a priori} on a specific embedding space (e.g., the unit interval).

Yet, a further essential challenge arises: \emph{Plasticity} (or structural adaptivity). 
In many realistic systems, the interaction network is not a fixed background; it evolves adaptively in response to the particles' states. 
A paradigmatic example is Hebbian learning in neuroscience, where ``cells that fire together, wire together''. 
From a mathematical perspective, this creates a tightly coupled feedback loop between the stochastic differential equations (SDEs) governing the particles and the random evolution of the graph weights themselves. 
The classical assumption of a static limit space (e.g., a fixed probability space indexing the continuum of particles) breaks down, as the structure of interactions is itself a random process being shaped by the flow.

The challenge of plasticity naturally invites a complementary perspective from the theory of \emph{dynamic random networks}.
In parallel to the developments in kinetic theory, significant progress has been made in characterizing the evolution of time-varying graphs  \cite{crane2016dynamic, ganguly2025mean, pensky2019dynamic}.
This duality offers a profound insight: the system can be viewed either as an interacting particle system depending on a dynamical interaction graph, or conversely, as a dynamic random network where the nodes are equipped with complex internal dynamics (such as SDEs).
In this article, we demonstrate that the mean-field approach from particle systems and the limit theory from dense graphs naturally converge to describe this coupled evolution.

\subsection{A prototypical model}

To isolate the fundamental challenges arising from the reciprocal feedback between states and weights dynamics, we focus on a prototypical model that couples standard diffusive particle dynamics with state-dependent weight evolution. For a discussion on more diverse dynamics, we refer to Section~\ref{subsubsec:diverse_dynamics}.

Specifically, we consider a system of $N$ agents where the state $\bm{X}_i(t) \in \mathbb{D}$ ($= \mathbb{R}^d$ or $\mathbb{T}^d$) and the interaction weight $\bm{W}_{i,j}(t) \in \mathbb{R}$ co-evolve according to the following coupled system of SDEs:
\begin{subequations}\label{eqn:main_SDEs}
\begin{equation}  \label{eqn:main_SDEs_X}
\rd \bm{X}_i = \mu(\bm{X}_i) \rd t + \nu \rd \bm{B}_{i} + \frac{1}{N} \sum_{j \in [N] \setminus \{i\}} \bm{W}_{i,j} \sigma(\bm{X}_i, \bm{X}_j) \rd t, \quad \forall i \in [N],
\end{equation}
\begin{equation} \label{eqn:main_SDEs_w}
\rd \bm{W}_{i,j} = \alpha(\bm{X}_i, \bm{X}_j) \bm{W}_{i,j} \rd t + \beta(\bm{X}_i, \bm{X}_j) \rd t, \quad \forall i,j \in [N], i \neq j.
\end{equation}
\end{subequations}
Here, $\mu: \mathbb{D} \to \mathbb{R}^d$ represents the drift (e.g., intrinsic dynamics or external force), $\sigma: \mathbb{D} \times \mathbb{D} \to \mathbb{R}^d$ is the interaction kernel, and $\nu \ge 0$ denotes the noise intensity associated with independent Brownian motions $(\bm{B}_i)_{i \in [N]}$. 
The weights $\bm{W}_{i,j}$ are treated as dynamic stochastic variables driven by the particle correlations via the coefficients $\alpha, \beta: \mathbb{D} \times \mathbb{D} \to \mathbb{R}$.

\subsection{From classical chaos to propagation of dissociatedness}
In the classical exchangeable mean-field theory ($\bm{W}_{i,j} \equiv 1$), one of the cornerstone concepts is \emph{propagation of chaos} \cite{kac1956foundations, sznitman1991topics}. Intuitively, as $N \to \infty$, the vanishing influence of individual particles effectively erases finite-$N$ correlations. This asymptotic decoupling, leading to \emph{i.i.d.~behavior}, reduces the macroscopic description to a single ``typical'' particle evolving via the McKean-Vlasov SDE, where the interaction sum is replaced by an expectation over the common law.

In our non-exchangeable system with adaptive weights, however, there is no single ``typical'' particle whose law describes the global dynamics.
Due to the reciprocal feedback described in \eqref{eqn:main_SDEs}, the state $\bm{X}_i(t)$ and the interaction weights $\bm{W}_{i,j}(t)$ are strongly coupled. 
Unlike the mean-field interaction which scales as $1/N$, this $O(1)$ local coupling remains non-vanishing as $N \to \infty$. 
Consequently, the pair $(\bm{X}_i, \bm{W}_{i,j})$ cannot be decoupled.

Nevertheless, a more refined structured independence survives. 
Borrowing concept from the theory of exchangeable arrays \cite{aldous1981representations, hoover1979relations, kallenberg1989representation}, we formalize this as:

\begin{defi}[Dissociatedness] \label{defi:dissociatedness}
Let $\bm{\mathcal{Y}} = \{\bm{Y}_\alpha\}_{\alpha \in \mathcal{I}}$ be a family of random variables indexed by finite ordered tuples of integers. 
For each index $\alpha = (i_1, \dots, i_k)$, let $\supp(\alpha) = \{i_1, \dots, i_k\} \subset \mathbb{N}$ denote the set of agents involved.
We say that $\bm{\mathcal{Y}}$ is \emph{dissociated} (or jointly dissociated) if, for any two disjoint sets of agents $I, J \subset \mathbb{N}$ (i.e., $I \cap J = \emptyset$), the sub-collection of variables supported on $I$, denoted by $\bm{\mathcal{Y}}_I \defeq \{\bm{Y}_\alpha : \supp(\alpha) \subseteq I\}$, is independent of the sub-collection supported on $J$, denoted by $\bm{\mathcal{Y}}_J \defeq \{\bm{Y}_\alpha : \supp(\alpha) \subseteq J\}$.
\end{defi}

In our context, the family $\bm{\mathcal{Y}}$ encapsulates the full \emph{state-weight structure} of the system, consisting of states indexed by singletons (i.e., $\bm{Y}_{i} = \bm{X}_i$) and weights indexed by ordered pairs (i.e., $\bm{Y}_{(i,j)} = \bm{W}_{i,j}$). 
Analogous to the classical mean-field picture, the finite system \eqref{eqn:main_SDEs} satisfies dissociatedness only asymptotically. 
Ideally, to filter out the vanishing correlations arising from $O(1/N)$ interactions, we construct the McKean-Vlasov SDEs \eqref{eqn:main_McKean_Vlasov} (detailed later), which yields a limiting process $(\overline{\bm{X}}, \overline{\bm{W}})$ satisfies Definition~\ref{defi:dissociatedness} \emph{exactly}. 
Within this ideal system, while $\overline{\bm{X}}_1$ and $\overline{\bm{W}}_{1,2}$ remain correlated due to the shared index $1$, recursive applications of the definition imply that $\overline{\bm{X}}_1$, $\overline{\bm{X}}_2$, and $\overline{\bm{W}}_{3,4}$ form a mutually independent set. 

Assuming sufficient regularity of the coefficients in \eqref{eqn:main_SDEs} and appropriate moment bounds on the initial data, the following theorem establishes the propagation of dissociatedness, proving that the finite system stays within $O(N^{-1/2})$ distance of this ideal dissociated limit.

\begin{thm}[Propagation of Dissociatedness] \label{thm:propagation_of_independent_mixture}
Assume that the coefficients of the system \eqref{eqn:main_SDEs} satisfy the regularity conditions in Assumption~\ref{assu:coefficients}.
Consider initial data $(\bm{X}_{(0)}, \bm{W}_{(0)})$ for a system of size $N$ that satisfies the dissociatedness property (Definition~\ref{defi:dissociatedness}) and the moment bounds in Assumption~\ref{assu:initial_data}.

Let $(\bm{X}(t), \bm{W}(t))$ be the unique solution (in the almost sure sense) to the finite system \eqref{eqn:main_SDEs} starting from initial data $(\bm{X}_{(0)}, \bm{W}_{(0)})$.
And let $(\overline{\bm{X}}(t), \overline{\bm{W}}(t))$ be the unique solution to the McKean-Vlasov SDEs \eqref{eqn:main_McKean_Vlasov} and \eqref{eqn:sampling_definition} in the sense of Proposition~\ref{prop:well_posedness_McKean_Vlasov} and Remark~\ref{rmk:finite_sampling}, starting from the same initial data and driven by the same stochastic forcing.

Then, the limiting process $(\overline{\bm{X}}(t), \overline{\bm{W}}(t))$ remains dissociated at the path level.
Moreover, there exists a constant $C(t) > 0$, depending on time, dimension, and the bounds in Assumptions~\ref{assu:coefficients} and \ref{assu:initial_data}, such that:
\begin{equation} \label{eqn:propagation_of_independent_mixture}
\sup_{i \in [N]} \E \Big[ |\bm{X}_i(t) - \overline{\bm{X}}_i(t)| \Big] + \sup_{i \neq j} \E \Big[ |\bm{W}_{i,j}(t) - \overline{\bm{W}}_{i,j}(t)| \Big] \leq \frac{C(t)}{\sqrt{N}}.
\end{equation}
\end{thm}

\subsection{Empirical stability}

The propagation of dissociatedness established in Theorem~\ref{thm:propagation_of_independent_mixture} serves as the non-exchangeable counterpoint to the classical propagation of chaos at the level of individual trajectories.

However, there is a more ``global'' aspect to the mean-field theory. 
It is well-understood that \emph{``the propagation of chaos property is equivalent to the convergence of the empirical process towards a deterministic limit''} \cite{chaintron2022propagation1,chaintron2022propagation2}. 
This allows one to track the evolution of the entire population as a single object (the one-particle distribution), with convergence quantified by distances such as the Wasserstein distance.

For our system \eqref{eqn:main_SDEs}, however, the naive ``empirical measure of states converging to one-particle distribution'' alone is insufficient to capture the dynamics. 
By projecting solely onto the state space, it discards the essential information of the interaction network, rendering the reconstruction of the full dynamics impossible.

To restore the non-exchangeable counterpoint of this equivalence, we utilize the Aldous-Hoover representation theorem to reduce the collection of dissociated variables into a functional representation (analogous to the classical setting, which can be understood as a de Finetti reduction for chaotic variables). 
We invoke this fundamental result, adapted to our setting, to identify the natural limit object:

\begin{prop}[Aldous-Hoover Representation] \label{prop:aldous_hoover}
Let $\bm{\mathcal{Y}}_N = ((\bm{X}_{i})_{i \in [N]}, (\bm{W}_{i,j})_{i,j \in [N]})$ be a dissociated array as in Definition~\ref{defi:dissociatedness}. 
Then there exists a latent probability space $(\Omega, \mathcal{A}, \pi)$ and measurable functions:
\begin{equation} \label{eqn:triplet_functions}
\mathsf{X}: \Omega \to \mathbb{D} \quad \text{and} \quad \mathsf{W}: \Omega \times \Omega \to \mathbb{R},
\end{equation}
such that the array is equal in law to the sampling of these functions via independent latent variables $\bm{\xi}_1, \dots, \bm{\xi}_N$ on $\Omega$:
\begin{equation} \label{eqn:representation_equality}
\bm{X}_i \stackrel{d}{=} \mathsf{X}(\bm{\xi}_i), \qquad \bm{W}_{i,j} \stackrel{d}{=} \mathsf{W}(\bm{\xi}_i, \bm{\xi}_j).
\end{equation}
Here, we allow the laws $\pi_i = \Law(\bm{\xi_i})$ to be non-identical (e.g., to accommodate deterministic initial data with $\Omega = [N]$), provided they satisfy the averaging condition $\frac{1}{N}\sum_{i=1}^N \pi_i = \pi$.
\end{prop}

Proposition~\ref{prop:aldous_hoover} is an application of an adapted Aldous-Hoover theorem\footnote{Standard Aldous-Hoover representation typically yields a kernel of the form $w(\xi_i, \xi_j, \eta_{i,j})$, where $\eta_{i,j}$ represents independent edge-specific randomness. However, in our context where the system size $N$ is fixed prior to the limit (and where we do not impose symmetry), one can technically absorb the row-specific edge noises into the node label by redefining $\tilde{\xi}_i \defeq (\xi_i, (\eta_{i,j})_{j \in [N]})$. This trick eliminates the explicit edge noise in the representation, yielding the form $w(\tilde{\xi}_i, \tilde{\xi}_j)$.} for dissociated arrays (see, e.g., \cite{kallenberg2005probabilistic}).
The functional representation $\mathcal{Y} = (\mathsf{X}, \mathsf{W})$ derived above characterizes the \emph{structural limit} of the system. 
This concept sharply contrasts with the classical \emph{distributional limit}, which captures only the state evolution (via empirical measures and one-particle distributions) at the expense of the interaction geometry.

To define the topology of the state-weight structure, one must recognize that the specific indexing, whether by discrete labels $i \in [N]$ or the more general variables $\xi \in \Omega$, serves merely as a representation rather than the essence of the dynamics. 
Consequently, while a theory dependent on specific labeling is technically conceivable, a natural definition must look beyond these artifacts to capture the intrinsic, \emph{representation-independent} structure.

This requirement mirrors the foundational distances in both constituent fields. 
In classical mean-field theory \cite{dobrushin1979vlasov, sznitman1991topics}, the typically used \emph{Wasserstein distance} measures the probability measures (which can be reduced from random variables such as $\mathsf{X}: \Omega \to \mathbb{D}$) by treating particles as indistinguishable. 
In graph limit theory \cite{borgs2008convergent, lovasz2012large}, the \emph{unlabeled cut distance} identifies graphs or \emph{graphons} (or $L^\infty$-kernels of form $\mathsf{W}: \Omega \times \Omega \to \mathbb{R}$) up to measure-preserving transformations. We also refer to \cite{janson2008graph} for a characterization of graphons via dissociated arrays.

Synthesizing these two perspectives, we introduce the distance $\delta_{W_1, \square}(\mathcal{Y}^{(1)},\mathcal{Y}^{(2)})$ (see Section~\ref{subsec:metric_def} for a more rigorous discussion).
Defined as the infimum over all couplings via measure-preserving maps $\varphi_k: \Omega \to \Omega^{(k)}$ bridging the distinct latent spaces, it takes the form:
\begin{equation} \label{eqn:informal_metric_def}
\delta_{W_1, \square}(\mathcal{Y}^{(1)},\mathcal{Y}^{(2)}) \coloneqq \inf_{\varphi_1, \varphi_2} \left( \|\mathsf{X}^{(1)}_{\varphi_1} - \mathsf{X}^{(2)}_{\varphi_2}\|_{L^1} + \|\mathsf{W}^{(1)}_{\varphi_1} - \mathsf{W}^{(2)}_{\varphi_2}\|_{\square} \right),
\end{equation}
where $\mathsf{X}^{(k)}_{\varphi_k}(\xi) \coloneqq \mathsf{X}^{(k)}(\varphi_k(\xi))$ and $\mathsf{W}^{(k)}_{\varphi_k}(\xi, \zeta) \coloneqq \mathsf{W}^{(k)}(\varphi_k(\xi), \varphi_k(\zeta))$ denote the rearranged state and interaction kernel, respectively.
Finally, $\|\cdot\|_{\square}$ denotes the cut norm, a fundamental concept in graph limit theory measuring the maximum interaction intensity over rectangular subsets.
We remark that its equivalence to the operator norm $\|\cdot\|_{L^\infty \to L^1}$ makes it the natural choice for controlling error propagation in the present setting:
\begin{equation*}
\|\mathsf{W}\|_{\square} \coloneqq \sup_{S,T \subseteq \Omega} \left| \int_{S \times T} \mathsf{W}(\xi, \zeta) \, \rd\pi(\xi) \rd\pi(\zeta) \right|.
\end{equation*}
Equipped with the distance $\delta_{W_1, \square}$, we establish our second main result: the stability of the empirical state-weight structure.

\begin{thm}[Empirical Stability of State-Weight Structure] \label{thm:empirical_data_stability}
Assume that the coefficients of the system \eqref{eqn:main_SDEs} satisfy the regularity conditions in Assumption~\ref{assu:coefficients}.
Let $\bm{\mathcal{Y}}_N(t) = ((\bm{X}_{i}(t))_{i \in [N]}, (\bm{W}_{i,j}(t))_{i,j \in [N]})$ denote the empirical structure of the finite system \eqref{eqn:main_SDEs}.
Let $\overline{\mathcal{Y}}(t) = (\overline{\mathsf{X}}(t), \overline{\mathsf{W}}(t))$ denote the structural limit evolving according to the McKean-Vlasov SDEs \eqref{eqn:main_McKean_Vlasov} in the sense of Proposition~\ref{prop:well_posedness_McKean_Vlasov}.
Assume that both initial configurations $\bm{\mathcal{Y}}_N(0)$ and $\overline{\mathcal{Y}}(0)$ satisfy the moment bounds specified in Assumption~\ref{assu:initial_data}.

Then, there exists a constant $C(t) > 0$, depending on time, dimension, and the bounds in Assumptions~\ref{assu:coefficients} and \ref{assu:initial_data}, such that:
\begin{equation} \label{eqn:empirical_data_stability}
\mathbb{E} \Big[ \delta_{W_1, \square}\big( \bm{\mathcal{Y}}_N(t), \overline{\mathcal{Y}}(t) \big) \Big] 
\le C(t) \, \mathbb{E} \Big[ \delta_{W_1, \square}\big( \bm{\mathcal{Y}}_N(0), \overline{\mathcal{Y}}(0) \big) \Big] + \frac{C(t)}{\sqrt{\log N}}.
\end{equation}
\end{thm}
This result serves as the non-exchangeable counterpoint to the celebrated quantitative stability estimates for the classical Vlasov equation. 
In the exchangeable setting ($\bm{W}_{i,j} \equiv 1$), under similar Lipschitz assumptions on the coefficients, it is well-known (see, e.g., \cite{bolley2011stochastic, dobrushin1979vlasov, sznitman1991topics}) that for any particle system $(\bm{X}_i(t))_{i \in [N]}$ satisfying \eqref{eqn:main_SDEs_X} and any solution $f$ to the corresponding Vlasov equation, the stability estimate takes the form:
\begin{equation} \label{eqn:classical_stability_review}
\E \bigg[ W_1 \bigg( \frac{1}{N}\sum_{i=1}^N \delta_{\bm{X}_i(t)} , f(t,\cdot)\bigg) \bigg]
\leq C(t) \E \bigg[ W_1 \bigg( \frac{1}{N}\sum_{i=1}^N \delta_{\bm{X}_i(0)} , f(0,\cdot)\bigg) \bigg] + C(t) N^{- \theta}.
\end{equation}
Our Theorem~\ref{thm:empirical_data_stability} recovers exactly this estimate in the geometry of state-weight structures.
This establishes a complete parallel: \emph{just as classical propagation of chaos is equivalent to the weak-* convergence of the empirical measure (typically quantified by the Wasserstein distance), the propagation of dissociatedness finds its natural ``global'' characterization as the convergence of the empirical structure in the topology induced by $\delta_{W_1, \square}$.}

The first term on the right-hand side of \eqref{eqn:empirical_data_stability} controls the dynamical amplification of the initial approximation error, analogous to the first term in \eqref{eqn:classical_stability_review}. 
The second term, $O(1/\sqrt{\log N})$, represents the cost of approximating the continuum limit by a finite sample. 
In the classical exchangeable setting, the corresponding rate $N^{-\theta}$ arises from the concentration of empirical measures in finite-dimensional spaces, where sampling $N$ independent points classically yields polynomial convergence in the Wasserstein distance \cite{boissard2011problemes, dobric1995asymptotics, fournier2015rate}.

In contrast, a general graph limit cannot be embedded exactly into a low-dimensional Euclidean space; consequently, the logarithmic rate is the optimal worst-case guarantee without further structural assumptions, reflecting the inherent complexity of the interaction structure. 
We formalize this sampling barrier in the following lemma, extending the sampling result of \cite{borgs2008convergent} to include state variables:
\begin{lem}[Sampling Lemma] \label{lem:sampling_lemma}
Let $\mathcal{Y} = (\mathsf{X}, \mathsf{W})$ be a state-weight structure represented on a latent probability space $(\Omega, \mathcal{A}, \pi)$.
Assume that $\mathsf{X} \in L^2(\Omega)$ and $\mathsf{W} \in L^\infty(\Omega \times \Omega)$.

Consider $N$ independent latent variables $\bm{\xi}_1, \dots, \bm{\xi}_N$ taking values in $\Omega$, whose laws $\pi_i = \Law(\bm{\xi}_i)$ are absolutely continuous with respect to $\pi$ and satisfy the averaging condition $\frac{1}{N}\sum_{i=1}^N \pi_i = \pi$.
Let $\bm{\mathcal{Y}}_N = ((\bm{X}_{i})_{i \in [N]}, (\bm{W}_{i,j})_{i,j \in [N]})$ be the discrete representation generated by sampling $\mathcal{Y}$ via these latent variables:
\begin{equation}
\bm{X}_i = \mathsf{X}(\bm{\xi}_i), \qquad \bm{W}_{i,j} = \mathsf{W}(\bm{\xi}_i, \bm{\xi}_j).
\end{equation}

Then, there exists a constant $C > 0$, such that the following approximation holds:
\begin{equation} \label{eqn:sampling_lemma}
\mathbb{E} \Big[ \delta_{W_1, \square}\big( \bm{\mathcal{Y}}_N, \mathcal{Y} \big) \Big] \leq \frac{C (\|\mathsf{X}\|_{L^2} + \|\mathsf{W}\|_{L^\infty})}{\sqrt{\log N}}.
\end{equation}
\end{lem}

This lemma critically allows us to translate the stability results from the continuous limit space back to the finite particle system in the proof of Theorem~\ref{thm:empirical_data_stability}. Furthermore, this result serves as a constructive proof of the \emph{compactness} of state-weight structures under moment constraints, as it guarantees that finite-dimensional approximations form a dense subset in the $\delta_{W_1, \square}$ metric.

\subsection{Organization of the article}
The remainder of this article is organized as follows. 
In Section~\ref{sec:discussion}, we situate our results within the broader literature, bridging perspectives from mean-field theory, dynamic random networks, and graph theory.
Section~\ref{sec:framework} rigorously constructs the mathematical framework, including the precise definition of the McKean-Vlasov SDEs and the $\delta_{W_1, \square}$ distance. 
Section~\ref{sec:proofs} establishes the well-posedness of the limit equations and provides the proofs for the propagation of dissociatedness (Theorem~\ref{thm:propagation_of_independent_mixture}) and the empirical stability (Theorem~\ref{thm:empirical_data_stability}). 
Finally, the proof of the Sampling Lemma (Lemma~\ref{lem:sampling_lemma}), which adapts combinatorial techniques to our setting, is deferred to Appendix~\ref{subsec:proof_sampling}.

\section{Discussion and Future Perspectives} \label{sec:discussion}

Our results bridge several active lines of research, connecting the kinetic theory of particle systems with the probability theory of exchangeable arrays and the limit theory of dense graphs. In this section, we discuss how our framework of structural limits relates to and extends these existing perspectives.

\subsection{Graph kernel mean-field limit and the fiberwise perspective} 
\label{subsec:kernel_fiberwise}

The past decade has witnessed rapid advancements in non-exchangeable mean-field theory. The use of graph limits to derive macroscopic equations (often PDEs) for such systems arguably dates back to works like \cite{chiba2018mean, kaliuzhnyi2018mean}, though the study of heterogeneous agent systems using other mathematical tools may have a longer history.

This framework proved highly effective for systems with heterogeneous yet \emph{non-adaptive} interaction structures. In the continuum limit, the heterogeneity is encapsulated by a fixed kernel $w: \Omega \times \Omega \to \R$, and the collective dynamics is characterized by an extended density function $f(t,x,\xi)$, where $\xi \in \Omega$ serves as a continuum generalization of the discrete agent labels $i \in [N]$. Intuitively, $f(t, \cdot, \xi) \in \mathcal{P}(\mathbb{D})$ represents the probability distribution of the ``typical particle'' associated with the label $\xi$. Taking the state dynamics in \eqref{eqn:main_SDEs_X} as an illustrative example (with fixed weights), approaches such as those in \cite{chiba2018mean, kaliuzhnyi2018mean} posit that the limit satisfies an extended Vlasov equation of the form:
\begin{equation} \label{eqn:multi_agent_Vlasov}
\begin{aligned}
\partial_t f(t,x,\xi) &+ \nabla_x \cdot \left( \left[ \mu(x) + \int_{\Omega} w(\xi,\zeta) \int_{\mathbb{D}} \sigma(x, y) f(t, y, \zeta) \, \rd y \rd \zeta \right] f(t, x, \xi) \right)
\\
&- \frac{\nu^2}{2} \Delta_x f(t,x,\xi) = 0.
\end{aligned}
\end{equation}
To justify this mean-field limit, the typical strategy involves first establishing that the discrete weights $w^{(n)}$, when embedded into the latent space $\Omega$, converge to the continuum kernel $w$. Building on this, one then adapts the foundational techniques of classical kinetic theory, specifically the characteristic flow stability developed by Neunzert \cite{neunzert1978mathematical, neunzert2006approximation}, Braun and Hepp \cite{braun1977vlasov}, Dobrushin \cite{dobrushin1979vlasov}, and Sznitman \cite{sznitman1991topics}, to establish the stability of $f$ in a fiberwise metric, such as an $L^p$-averaged Wasserstein distance:
\begin{equation} \label{eqn:fiberwise_extension}
\begin{aligned}
d_F \left( f^{(n)}(t), f(t) \right) \coloneqq \left[ \int_{\Omega} \Big( d_{\mathcal{P}} \big( f^{(n)}(t, \cdot, \xi), f(t, \cdot, \xi) \big) \Big)^p \rd \pi(\xi) \right]^{1/p}.
\end{aligned}
\end{equation}
Early contributions, such as \cite{chiba2018mean, kaliuzhnyi2018mean}, typically required $w^{(n)} \to w$ in the strong $L^1$ topology. It was later recognized in works such as \cite{bayraktar2023graphon, bet2024weakly, jabin2024non} that the \emph{cut norm} topology (and the associated theory of graphons \cite{lovasz2012large}) provides a more natural framework. Essential to this approach is that the cut norm is critical in controlling the operator norm of the interaction term, thereby avoiding the need for stronger norms and robustly handling random fluctuations that vanish in the large $N$ limit (e.g., in Erd\H{o}s-R\'enyi type random graphs).

From the perspective of our structural approach, the above framework benefits from the non-adaptive nature of the interactions and a two-scale structure: conceptually, each label $\xi$ indexes a ``local'' population of ``nearly identical'' particles (forming a local fiber); then, these labels form a stable graph continuum. Under this assumption, the subsequent dynamics for the state distribution $f$ primarily involve advection and diffusion in the state variable $x$, while the label $\xi$ remains a static parameter without further evolution.

For empirical data $((\bm{X}_{i})_{i \in [N]}, (\bm{W}_{i,j})_{i,j \in [N]})$, however, this stratification is often artificial. In a finite system, it is not determined \emph{a priori} how particles $i$ should be clustered together to shape a specific label $\xi$. Without such clustering, one is forced to compare a discrete particle against a potentially dispersed distribution. Notably, \cite{jabin2024non} addressed this difficulty by carefully employing a weak-* topology on the label space $\Omega = [0,1]$ to manage the aggregation of particles with distinct indices $i$ (and hence distinct embedded positions $\xi_i$). Such flexibility is absent in the rigid fiberwise definition \eqref{eqn:fiberwise_extension}.
Nevertheless, this issue becomes increasingly critical when the weights $\bm{W}$ are adaptive: even if such fibers could be identified initially, the feedback between states and weights would cause particles within the same fiber to diverge. As the interaction profiles evolve stochastically, the fiber structure effectively ``splits'', rendering any description based on fixed coordinates $\xi$ insufficient to capture the full dynamics.

Crucially, the limit described by $f(t, x, \xi)$ and $w(\xi, \zeta)$ constitutes a special case of our structural limit. Such a system satisfies the dissociatedness property. Specifically, as mentioned in our companion work \cite{zhou2024coupling}, one can explicitly construct its Aldous-Hoover representation $(\mathsf{X}, \mathsf{W})$ on an extended latent space $\Omega_f \defeq \mathbb{D} \times \Omega$ equipped with the measure $\pi_f(\rd x, \rd \xi) = f(t, \rd x, \xi) \rd \xi$. The state and weight functions are then given by the projection $\mathsf{X}(x, \xi) = x$ and the pullback $\mathsf{W}((x, \xi), (y, \zeta)) = w(\xi, \zeta)$.
However, when the clustering is obscure or dynamically evolving as in the adaptive setting, the Aldous-Hoover representation and the propagation of dissociatedness provide a more robust formalism.

We emphasize, however, that not all developments in non-exchangeable mean-field theory can be readily reformulated within this probabilistic framework.
Addressing very sparse interaction regimes often employs alternative graph limits based on singular measures on the latent space. These structures entail probability-zero connections incompatible with standard probability theory of exchangeable arrays, and the approximation by finite systems is typically established in a more restricted and delicate sense.
For example, \cite{gkogkas2022graphop} analyzed graphops (objects introduced in \cite{backhausz2022action}), while \cite{kuehn2022vlasov} considered a notion of digraph measures. Both works established convergence $f^{(n)} \to f$ (effectively utilizing $p=\infty$ in \eqref{eqn:fiberwise_extension}) by employing ad hoc discretizations to ensure stronger weight convergence in their respective spaces. 
In parallel, \cite{jabin2025mean, jabin2023mean} considered extended graphons based on weak-* definitions, utilizing techniques akin to the Szemerédi Regularity Lemma \cite{szemeredi1975sets, szemeredi1978regular} to guarantee that the limit of their hierarchical estimates is identified by a kernel.

Finally, it is worth noting that the kernel (graph limit) based approach has demonstrated significant versatility, having been extended to encompass mean-field games \cite{aurell2022stochasticII, caines2021graphon, carmona2022stochasticI, lacker2023label, parise2023graphon}, higher-order nonlinear dependencies on the measure \cite{coppini2025nonlinear, crucianelli2024interacting}, and hypergraphs \cite{ayi2024mean, kuehn2025vlasov}.

\subsection{Dynamic random networks and probabilistic foundations}
\label{subsec:dynamic_networks}

In parallel to the developments in kinetic theory, there has been an approach to constructing probabilistic descriptions of network dynamics based on the classical results of Aldous \cite{aldous1981representations}, Hoover \cite{hoover1979relations}, and Kallenberg \cite{kallenberg1989representation}.

The rigorous study of exchangeable Markov processes integrated with graph limit theory was arguably initiated in \cite{crane2016dynamic, crane2017exchangeable}. By leveraging induced subgraph densities as fundamental observables, these works used reverse martingale convergence and projective consistency to establish the convergence of finite-dimensional distributions to the continuum limit.
This analysis further revealed that the limiting sample paths possess locally bounded variation, and for Feller processes, the generator admits a decomposition into global rewiring, vertex-level updates, and independent edge flips (corresponding to classical Erd\H{o}s-R\'enyi dynamics that manifest as deterministic flows or vanishing noise in the mean-field limit). This theoretical framework was subsequently extended to more general exchangeable arrays in \cite{cerny2020markovian}.

Within the same probabilistic framework and the perspective of induced subgraph densities, recent studies \cite{athreya2021graphon, braunsteins2022graphon, braunsteins2023sample} have deeply explored the specific mechanism of ``vertex-level fluctuations''. Motivated by population genetics models (e.g., via mutation and resampling) where the graph structure depends one-way on the complex evolution of the particle dynamics, these works established functional limit theorems in the Skorokhod topology. Furthermore, the complex vertex-level dynamics lead to a loss of the Markov property on the space of graphs, thereby allowing for limiting paths of unbounded variation beyond the scope of \cite{crane2016dynamic}.

From the perspective of our structural approach, the simultaneous convergence of induced subgraph densities (or graph homomorphism densities) considered in \cite{crane2016dynamic, crane2017exchangeable} is topologically equivalent to convergence in the unlabeled cut distance \cite{lovasz2012large}. Our framework addresses the fully coupled setting where complex vertex-level SDEs and edge-level dynamics exert a \emph{two-way} influence. In this context, the unified metric $\delta_{W_1, \square}$ provides a natural topology for dynamics propagating dissociatedness. 
As both approaches rely fundamentally on the theory of exchangeable arrays, the representation theory implies that for specific models defined by concrete rules, the limit identified through our explicit coupling coincides with the definition obtained via reverse martingales and projective consistency.
Furthermore, from our viewpoint, leveraging the propagated dissociatedness (specifically \eqref{eqn:propagation_of_independent_mixture}) readily yields quantitative estimates for graph homomorphism densities (e.g., $\mathbb{E}|t(F,\bm{W}^{(N)}) - t(F,\mathsf{W})| \sim O(N^{-\theta})$ for any $\theta < 1/2$ and for any finite graph $F$; see \cite[Theorem 10.3]{lovasz2012large} for the concentration of homomorphism densities.)

Finally, we note that the study of dynamic graphs extends beyond the frameworks discussed above. For a more combinatorial perspective, we refer to \cite{garbe2024flip} and the references therein, which analyze ``flip processes'' where edge updates are governed by replacement rules on random induced subgraphs, rather than following a ``Hebbian'' dynamic determined solely by the activity of the edge itself or its constituent vertices.

\subsection{Co-evolutionary dynamics and limit characterizations}
\label{subsec:coevolution_limits}

The approaches discussed above may intersect in the setting of a genuine co-evolution (or two-way feedback) between the interaction (edge) structure and particle (vertex) states.

In this context, the mean-field description via an extended Vlasov-type PDE driven by a graph kernel generally fails, as adaptive weights cause agents with identical labels to diverge, breaking the fiber structure. Nevertheless, a kernel-based limit can still be recovered under specific structural assumptions, such as ``gravity-like'' unary interactions. In such scenarios, where weights depend solely on the target agent, the interaction field is determined by the collective distribution, allowing the system to be described by population dynamics (see, e.g., \cite{ayi2021mean, ben2024graph}; see also \cite{ayi2024mean} for a recent review).

Moving to general pairwise (binary) adaptive interactions, a kernel-based continuum limit can also be established in the fully deterministic setting (e.g., when $\nu=0$ in \eqref{eqn:main_SDEs}). Specifically, continuum limits characterized by measure-valued differential equations have been investigated in \cite{gkogkas2023continuum, gkogkas2025mean}. From our perspective, this result is compatible with the probabilistic structure: for the deterministic dynamics, the Aldous-Hoover representation is not forced to increase in complexity with respect to the filtration process over time, allowing the graph limit to be naturally captured by an evolving kernel on a fixed latent space.
Applying this perspective to our system \eqref{eqn:main_SDEs} in the noiseless limit ($\nu = 0$), the continuum dynamics can be described by a coupled system of integro-differential equations. Using $\mathsf{X}(t, \xi)$ to denote the particle state at label $\xi \in \Omega$ and $\mathsf{W}(t, \xi, \zeta)$ for the evolving kernel, we present the limiting equation as in \cite{gkogkas2023continuum, gkogkas2025mean} in a pointwise form for simplicity (rather than the proper weak formulation):
\begin{equation}
\label{eqn:deterministic_limit}
\left\{
\begin{aligned}
\partial_t \mathsf{X}(t, \xi) &= \mu\big(\mathsf{X}(t, \xi)\big) + \int_{\Omega} \mathsf{W}(t, \xi,\zeta) \, \sigma\big(\mathsf{X}(t, \xi), \mathsf{X}(t, \zeta)\big) \, \mathrm{d} \zeta, \\
\partial_t \mathsf{W}(t, \xi,\zeta) &= \alpha\big(\mathsf{X}(t, \xi), \mathsf{X}(t, \zeta) \big) \mathsf{W}(t, \xi,\zeta) + \beta\big(\mathsf{X}(t, \xi), \mathsf{X}(t, \zeta) \big).
\end{aligned}
\right.
\end{equation}
We note that the analytic approach in \cite{gkogkas2023continuum, gkogkas2025mean} has some specific advantages. As discussed at the end of Section~\ref{subsec:kernel_fiberwise}, by stepping away from a full probabilistic description and employing a more cautious notion of convergence from the finite system, their limit process can be defined on more singular kernels (such as digraph measures or graphops). Capturing such singular structures is generally out of reach for the straightforward probabilistic characterization provided by the Aldous-Hoover representation.

Probabilistic frameworks that accommodate intrinsic stochasticity (where the filtration complexity generally precludes a fixed latent space representation) are exemplified by studies such as \cite{bayraktar2021mean, ganguly2025mean}. Both works characterize the asymptotic independence structure, or dissociatedness in our terminology, by examining the joint distribution of a fixed number of $k$ agents sampled from the total population $N$ (denoted as particle trajectories $\bm{X}_{1:k}^{(N)}$ and corresponding edges $\bm{W}_{1:k, 1:k}^{(N)}$) as the system size $N \to \infty$.
From our perspective, for all $k$, the resulting limit $(\overline{\bm{X}}_{1:k}, \overline{\bm{W}}_{1:k})$ is consistently generated from a unified limit structure (characterized by the Aldous-Hoover representation $(\mathsf{X}, \mathsf{W})$ via i.i.d. sampling of latent variables $\bm{\xi}_1, \dots, \bm{\xi}_k$). This corresponds to the classical picture where all $k$-marginals $\Law(\overline{\bm{X}}_{1:k})$ are not only tensorized but specifically arise as the $k$-fold tensor product $f^{\otimes k}$ of a single common distribution $f$.

Specifically, \cite{bayraktar2021mean} investigates a continuous-time interacting particle system where both particle states and interaction weights take values in a discrete space $\mathbb{Z}$ (effectively modeling dynamic multi-colored graphs). Adopting a pathwise perspective, they derived quantitative trajectory convergence rates of order $N^{-1/2}$ via Sznitman-type coupling and further studied Central Limit Theorems. In contrast, \cite{ganguly2025mean} operates in a discrete-time setting and recursively derives the property of asymptotic dissociatedness over time steps, while also providing a graphon-level analysis of the limiting network that characterizes the convergence of $\bm{W}^{(N)}$ in $\delta_{\square}$ via graph homomorphism densities. In this light, the convergence of $\bm{W}^{(N)}$ discussed in \cite{ganguly2025mean}, when viewed in conjunction with the convergence of empirical $\bm{X}^{(N)}$, corresponds to the full $\delta_{W_1, \square}$ convergence. Together with the trajectory-level analysis in \cite{bayraktar2021mean}, these findings across distinct models illuminate the complementary global and local facets of the propagation of dissociatedness.

\subsection{Relation to the hierarchical approach}
\label{subsec:coupling_and_tensorization}

Finally, we discuss an alternative perspective based on hierarchical techniques, a strategy well-established in classical kinetic theory (e.g., BBGKY hierarchies) and recently adapted to the non-exchangeable setting (see, e.g., \cite{jabin2025mean, jabin2023mean, zhou2024coupling}).

These hierarchies naturally generalize both the $k$-particle marginals from kinetic theory and the graph homomorphism densities (or subgraph densities) from graph limit theory.
Depending on the vantage point, they can be viewed as \emph{weighted marginals}, where the empirical distribution is modulated by the interaction weights, or as \emph{tensorized homomorphism densities}, effectively embedding the particle dynamics into a tensor product space indexed by the graph topology.

We present here the formulation streamlined from \cite{zhou2024coupling}, which systematizes related definitions evolving through \cite{jabin2025mean, jabin2023mean}.
Formally, for an oriented graph $F$ and a state-weight structure $\mathcal{Y} = (\mathsf{X}, \mathsf{W})$ on $(\Omega, \mathcal{A}, \pi)$, the observable $\tau(F, \mathcal{Y}) \in \mathcal{M}(\mathbb{D}^{\mathsf{v}(F)})$ is defined as:
\begin{equation} \label{eqn:tensorized_observable}
\tau(F, \mathcal{Y}) \coloneqq \int_{\Omega^{\mathsf{v}(F)}} \left( \prod_{(u,v) \in \mathsf{e}(F)} \mathsf{W}(\xi_u, \xi_v) \right) \bigotimes_{u \in \mathsf{v}(F)} \delta_{\mathsf{X}(\xi_u)} \, \prod_{u \in \mathsf{v}(F)} \rd\pi(\xi_u) .
\end{equation}
Crucially, these observables encode the intrinsic structure of the limit system.
As established in our companion work \cite{zhou2024coupling}, following the classical sampling and concentration inequality strategy of \cite{borgs2008convergent}, the convergence of these observables characterizes the topology induced by the metric.
Specifically, for a sequence $\mathcal{Y}_n$ and a limit $\mathcal{Y}$, the following equivalence holds (rigorously justified in \cite{zhou2024coupling} for the torus case $\mathbb{D} = \mathbb{T}$):
\begin{equation*}
\begin{aligned}
\tau(F,\mathcal{Y}_n) \xrightharpoonup{*} \tau(F,\mathcal{Y}), \ \forall \text{ simple oriented graphs } F
\quad \Longleftrightarrow \quad \mathcal{Y}_n \xrightarrow{\delta_{W_1, \square}} \mathcal{Y}.
\end{aligned}
\end{equation*}
This correspondence permits us to probe the intrinsic structural properties of $\mathcal{Y}$ without directly solving the optimization problem defining $\delta_{W_1, \square}$.

A computationally more tractable alternative arises if one restricts the hierarchy $\tau(F,\mathcal{Y})$ to \emph{acyclic} graphs (trees or forests); viewed through the lens of structural couplings, this amounts to a \emph{convex relaxation} of the optimization problem associated to $\delta_{W_1, \square}$.

Formally, this relaxation replaces the rigid rearrangements (measure-preserving maps) with coupling measures, shifting from point-to-point mappings to diffuse, ``Markovian'' correspondences (allowing mass splitting) convexifies the optimization domain. To quantify the discrepancy of interactions under such generalized couplings, we lift the kernels to linear operators.
Consider two structures $\mathcal{Y}^{(k)} = (\mathsf{X}^{(k)}, \mathsf{W}^{(k)})$ defined on probability spaces $(\Omega^{(k)}, \pi^{(k)})$ for $k \in \{1,2\}$.
For a coupling measure $\gamma \in \Pi(\pi^{(1)}, \pi^{(2)})$, we first define the forward coupling operator $T_\gamma: L^\infty(\Omega^{(2)}) \to L^1(\Omega^{(1)})$ via:
\begin{equation*}
T_\gamma(\psi)(\xi) \coloneqq \int_{\Omega^{(2)}} \psi(\zeta)\, \gamma(\xi, \rd \zeta), \quad \text{for a.e. } \xi \in \Omega^{(1)}.
\end{equation*}
Analogously, let $T_\gamma^*: L^\infty(\Omega^{(1)}) \to L^1(\Omega^{(2)})$ denote the adjoint coupling operator of $T_\gamma$. We further define the standard adjacency operator $T_{\mathsf{W}^{(k)}}: L^p(\Omega^{(k)}) \to L^p(\Omega^{(k)})$ (valid for any $p \in [1, \infty]$) associated with the kernel $\mathsf{W}^{(k)} \in L^\infty(\Omega^{(k)} \times \Omega^{(k)})$.
The distance $\gamma_{W_1, \square}$ is then defined by optimizing the transport cost alongside the operator commutation error:
\begin{equation} \label{eqn:gamma_distance}
\begin{aligned}
\gamma_{W_1, \square}(\mathcal{Y}^{(1)}, \mathcal{Y}^{(2)}) & \coloneqq \inf_{\gamma \in \Pi(\pi^{(1)}, \pi^{(2)})} \bigg[ \int_{\Omega^{(1)} \times \Omega^{(2)}} |\mathsf{X}^{(1)}(\xi) - \mathsf{X}^{(2)}(\zeta)| \, \gamma(\rd \xi, \rd \zeta)
\\
& \hspace{0.5cm} + \| T_{\mathsf{W}^{(1)}} \circ T_\gamma - T_\gamma \circ T_{\mathsf{W}^{(2)}} \|_{\text{op}} + \| T_\gamma^* \circ T_{\mathsf{W}^{(1)}} - T_{\mathsf{W}^{(2)}} \circ T_\gamma^* \|_{\text{op}} \bigg],
\end{aligned}
\end{equation}
where $\|\cdot\|_{\text{op}}$ denotes the $L^\infty \to L^1$ operator norm (e.g., in the first term, the difference is an operator mapping $L^\infty(\Omega^{(2)})$ to $L^1(\Omega^{(1)})$).

While the Monge-to-Kantorovich relaxation is typically gap-free in classical optimal transport, here it incurs a strict topological loss.
The reliance on diffuse couplings prevents the tracking of precise vertex identities needed to verify a ``return to the origin''.
This renders the $\gamma_{W_1, \square}$ topology blind to cycles (even in the unoriented sense) and thus strictly weaker than $\delta_{W_1, \square}$.
Nevertheless, exact topological equivalence is recovered precisely when the hierarchy is restricted to structures \emph{devoid of cycles} (rigorously justified in \cite{zhou2024coupling} for the torus case $\mathbb{D} = \mathbb{T}$):
\begin{equation*}
\begin{aligned}
\tau(F,\mathcal{Y}_n) \xrightharpoonup{*} \tau(F,\mathcal{Y}), \ \forall \text{ oriented forests } F
\quad \Longleftrightarrow \quad \mathcal{Y}_n \xrightarrow{\gamma_{W_1, \square}} \mathcal{Y}.
\end{aligned}
\end{equation*}
The pure graph-theoretic analogue of this phenomenon, known as \emph{fractional isomorphism} (or fractional overlays), has been explicitly studied in \cite{boker2021graph, dell2018lovasz, grebik2022fractional}.

Beyond establishing the structural equivalence, \cite{zhou2024coupling} integrates the hierarchical estimation techniques matured through \cite{jabin2025mean, jabin2023mean}, offering an approach complementary to the present work.
Specifically, it establishes the mean-field limit with respect to the $\gamma$-topology (an empirical stability analogous to \eqref{eqn:empirical_data_stability} rigorously justified for non-adaptive weights).
Conversely, while the trajectory-based coupling methods employed herein yield sharper convergence rates, they do not appear directly applicable to the $\gamma$-topology, as the introduction of diffuse couplings compromises the precise error propagation estimates.

\subsection{Future perspectives} \label{subsec:future_work}

We conclude by outlining several directions we plan to explore in future work.

\subsubsection{Sparser interaction weights}
A natural first question concerns the generalization to unbounded weights, such as $L^p$ graphons (for $p > 1$) \cite{borgs2019Lp, borgs2018Lp}. While technically demanding, we expect such extensions to follow the conceptual roadmap of the present dense framework.

A far more profound challenge arises in regimes of ``singular interaction structures'', exemplified by the graphops of \cite{gkogkas2022graphop, gkogkas2023continuum}, the digraph measures of \cite{gkogkas2025mean, kuehn2022vlasov}, and the extended graphons in \cite{jabin2025mean, jabin2023mean}.
In such settings, the probabilistic intuition underpinning the Sampling Lemma begins to falter, and the concentration mechanisms characterizing dense limits degenerate.
Consequently, establishing a rigorous limit theory likely requires retreating to more fundamental principles, specifically the Szemerédi Regularity Lemma \cite{szemeredi1975sets, szemeredi1978regular}.
Reconstructing the quantitative stability framework without the luxury of standard concentration arguments remains a formidable frontier to be navigated step by step.

\subsubsection{Multi-body interactions and hypergraphs}
While the present work focuses on pairwise interactions, there is a growing interest in extending mean-field theories to systems with multi-body interactions, modeled by hypergraphs \cite{ayi2024mean, kuehn2025vlasov}. From a probabilistic perspective, the interaction tensors in such systems can be viewed as higher-order exchangeable arrays.

We anticipate that the sampling techniques and convergence analysis developed here could, with sufficient care, be extended to the hypergraph setting. In particular, the propagation of dissociatedness is expected to persist. However, the global characterization would be less trivial: the analogue of the cut norm and the associated regularity or sampling lemmas become significantly more intricate and less canonical in the hypergraph regime. We refer to \cite{janson2008graph}, which elucidates how graphon limits could be induced from the structure of dissociated arrays, as a conceptual starting point for identifying the natural limit objects in such future explorations.

\subsubsection{Diverse dynamics for particles and weights} \label{subsubsec:diverse_dynamics}

While the prototypical model considered in this work focuses on the diffusive particle dynamics given by \eqref{eqn:main_SDEs_X}, the proposed framework is expected to be applicable to a broader class of kinetic models. Although providing an exhaustive survey is beyond the scope of this discussion, we note that integrate-and-fire dynamics, where interactions are mediated by Poisson events (jumps), serve as a widely-considered alternative to diffusive processes in the context of mathematical neuroscience. We refer to \cite{jabin2024non, jabin2023mean} where variations of \eqref{eqn:main_SDEs_X} in this direction are considered.

Another direction, specific to systems with adaptive weights, is to incorporate more diverse dynamical rules for the weight evolution.
For instance, the deterministic adaptation \eqref{eqn:main_SDEs_w} could be replaced by stochastic alternatives.
The first natural extension is the addition of diffusive noise, representing fluctuations in the link strength:
\begin{equation} \label{eqn:diffusive_w}
\rd \bm{W}_{i,j} = \alpha(\bm{X}_i, \bm{X}_j) \bm{W}_{i,j} \rd t + \beta(\bm{X}_i, \bm{X}_j) \rd t + \kappa \rd \bm{B}_{i,j},
\end{equation}
where $(\bm{B}_{i,j})$ are independent Brownian motions.
A second, distinct regime corresponds to dynamic random graphs (or binary opinion dynamics), where edges flip between discrete states $\{0,1\}$ via Poissonian jumps:
\begin{equation} \label{eqn:jump_w}
\rd \bm{W}_{i,j} = - \bm{W}_{i,j} \, \rd \bm{N}^{\text{off}}_{i,j} + (1 - \bm{W}_{i,j}) \, \rd \bm{N}^{\text{on}}_{i,j}.
\end{equation}
Here, $\bm{N}^{\text{off}}_{i,j}$ and $\bm{N}^{\text{on}}_{i,j}$ are Poisson processes with state-dependent intensities $\tilde{\alpha}(\bm{X}_i, \bm{X}_j)$ (controlling the $1 \to 0$ decay) and $\tilde{\beta}(\bm{X}_i, \bm{X}_j)$ (controlling the $0 \to 1$ creation), respectively.

Rigorous trajectory-wise control of such edge-specific stochasticity requires a refined analysis.
Ideally, one must characterize the propagation of dissociatedness via the full version of the Aldous-Hoover representation:
\begin{equation*}
\mathsf{X}: \Omega \to \mathbb{D} \quad \text{and} \quad \mathsf{W}: \Omega \times \Omega \times \tilde \Omega \to \mathbb{R},
\end{equation*}
where the additional latent dimension $\eta \in \tilde \Omega$ explicitly encodes the independent noise associated with each edge.
However, as elucidated in the literature of dynamic graphs \cite{crane2016dynamic, crane2017exchangeable}, such independent edge fluctuations typically average out in the continuum limit, manifesting effectively as deterministic flows for the interaction kernel.
For instance, in the binary flip model \eqref{eqn:jump_w}, the macroscopic evolution of the connection probability is expected to follow the deterministic structure \eqref{eqn:main_SDEs_w}, with effective drift coefficients derived from the jump intensities:
\begin{equation*}
\alpha(\cdot) = - \big( \tilde \alpha(\cdot) + \tilde \beta(\cdot) \big) \quad \text{and} \quad \beta(\cdot) = \tilde \beta(\cdot).
\end{equation*}
This suggests that the core stability results of this article should extend to these stochastic regimes, provided the microscopic noise remains strictly independent across edges.
Although the diffusive model \eqref{eqn:diffusive_w} introduces unbounded weights (linking to the $L^p$ regime discussed above), the physical intuition remains straightforward: independent microscopic noise, whether continuous or discrete, is expected to vanish in the mean-field limit due to the law of large numbers.
We therefore anticipate the following result:

\begin{conj}[Mean-field weight noise does not affect the limit] \label{conj:empirical_data_stability_limit}
Adopt the setting of Theorem~\ref{thm:empirical_data_stability}, but assume the $N$-particle system $\bm{\mathcal{Y}}_N(t) = ((\bm{X}_{i}(t))_{i \in [N]}, (\bm{W}_{i,j}(t))_{i,j \in [N]})$ evolves according to the stochastic dynamics \eqref{eqn:main_SDEs_X} coupled with the noisy weight updates \eqref{eqn:diffusive_w} or \eqref{eqn:jump_w}.
Let $\overline{\mathcal{Y}}(t) = (\overline{\mathsf{X}}(t), \overline{\mathsf{W}}(t))$ denote the structural limit solution to the McKean-Vlasov SDEs \eqref{eqn:main_McKean_Vlasov} in the sense of Proposition~\ref{prop:well_posedness_McKean_Vlasov} (corresponding to the original dynamics \eqref{eqn:main_SDEs_X}-\eqref{eqn:main_SDEs_w}).

Then, the quantitative stability estimates of this article remain valid. Specifically, there exists a constant $C(t)$, depending on time, dimension, and the bounds in Assumptions~\ref{assu:coefficients} and \ref{assu:initial_data}, such that:
\begin{equation*}
\begin{aligned}
\mathbb{E} \Big[ \delta_{W_1, \square}\big( \bm{\mathcal{Y}}_N(t), \overline{\mathcal{Y}}(t) \big) \Big]
\le C(t) \, \mathbb{E} \Big[ \delta_{W_1, \square}\big( \bm{\mathcal{Y}}_N(0), \overline{\mathcal{Y}}(0) \big) \Big] + \frac{C(t)}{\sqrt{\log N}}.
\end{aligned}
\end{equation*}
\end{conj}

This perspective naturally generalizes to weights taking values in a larger finite set, such as $\{-1, 0, 1\}$, where the transition intensities between any two states are governed by distinct functions.
However, in this regime, a single scalar graphon limit becomes insufficient due to the loss of distributional information.
For instance, a ``mixed zero'' resulting from an equal mixture of $+1$ and $-1$ states may exhibit totally different transition dynamics compared to a ``native zero'' state.
To avoid such ambiguities, the limit must be characterized by a \emph{multi-color graphon} that simultaneously tracks the density of each discrete state (e.g., by adapting the definition of $\delta_{W_1, \square}$ to control separate limits for $\mathsf{W}_+$ and $\mathsf{W}_-$).

A conceptually identical, yet analytically more demanding, challenge arises when the weights evolve continuously but with highly non-linear dependence on their own value:
\begin{equation} \label{eqn:nonlinear_w}
\rd \bm{W}_{i,j} = \alpha(\bm{X}_i, \bm{X}_j, \bm{W}_{i,j}) \rd t.
\end{equation}
In this setting, standard graphon convergence (in the cut norm) $\mathsf{W}_n \to \mathsf{W}$ is insufficient, as it generally does not imply the convergence of the drift term $\alpha(\cdot, \mathsf{W}_n) \to \alpha(\cdot, \mathsf{W})$, thereby precluding the propagation of cut norm estimates via Gronwall's inequality for $t > 0$.
A rigorous limit theory in this context would likely necessitate a topology strong enough to preserve observables of the form $\beta(\mathsf{W}_n) \to \beta(\mathsf{W})$ for arbitrary non-linear functions $\beta: \mathbb{R} \to \mathbb{R}$, conceptually connecting this problem to the theory of renormalized solutions for transport equations (DiPerna-Lions theory \cite{diperna1989ordinary}).

\subsubsection{Fast adaptation limit}
Another related direction explores the regime of time-scale separation, where the weights evolve significantly faster than the particle positions.
Consider the following singularly perturbed system, governed by a small parameter $\epsilon \ll 1$:
\begin{subequations}\label{eqn:decay_SDEs}
\begin{equation} \label{eqn:decay_SDEs_X}
\rd \bm{X}^\epsilon_i = \mu(\bm{X}^\epsilon_i) \rd t + \nu \rd \bm{B}_{i} + \frac{1}{N} \! \sum_{j \in [N] \setminus \{i\}} \bm{W}^\epsilon_{i,j} \sigma(\bm{X}^\epsilon_i, \bm{X}^\epsilon_j) \rd t,
\end{equation}
\begin{equation} \label{eqn:decay_SDEs_w}
\rd \bm{W}^\epsilon_{i,j} = \epsilon^{-1} \big(-\bm{W}^\epsilon_{i,j} + \beta(\bm{X}^\epsilon_i, \bm{X}^\epsilon_j) \big) \rd t.
\end{equation}
\end{subequations}
As $\epsilon \to 0$, the fast variables $\bm{W}^\epsilon_{i,j}$ rapidly relax to the quasi-steady state given by $\beta(\bm{X}^\epsilon_i, \bm{X}^\epsilon_j)$.
Consequently, the coupled system \eqref{eqn:decay_SDEs} is expected to converge to a limit with an effective interaction kernel $K(x,y) = \beta(x,y)\sigma(x,y)$:
\begin{equation} \label{eqn:decay_SDEs_limit}
\rd \bm{X}_i = \mu(\bm{X}_i) \rd t + \nu \rd \bm{B}_{i} + \frac{1}{N} \! \sum_{j \in [N] \setminus \{i\}} \beta(\bm{X}_i, \bm{X}_j) \sigma(\bm{X}_i, \bm{X}_j) \rd t.
\end{equation}
This limit bridges our model with the classical exchangeable mean-field theory.
Analogous singular limits have been previously analyzed in the context of discrete states and weights in \cite{bayraktar2021mean}.
However, the extension to continuous collective motion models, such as the Cucker-Smale model \cite{bolley2011stochastic, gerber2025uniform, ha2009simple}, remains an interesting avenue for future research.
From this viewpoint, the adaptive dynamics \eqref{eqn:decay_SDEs} can be interpreted as a ``delayed'' Cucker-Smale system, where the alignment strength is not determined instantaneously by the current configuration but adapts based on the particles' interaction history.
Investigating these memory effects could provide new insights into both the non-exchangeable systems and the fundamental mechanisms of collective motion.

\subsubsection{Partial homogenization via adaptation} There is one more interesting direction. To build the intuition, it is helpful to consider a toy model:

Let $N \geq 1$ be a multiple of 6. For simplicity, we consider only the deterministic version ($\nu=0$) of the system \eqref{eqn:decay_SDEs} with weight linear relaxation parameter $\epsilon = 1$ fixed. Consider two systems with identical initial states, $X^{(1)}_i(0) = X^{(2)}_i(0)$ for all $i \in [N]$, which exhibit a periodic structure with period $N/6$: 
\begin{equation*}
\begin{aligned}
X^{(1)}_i(0) = X^{(1)}_{i+N/6}(0), \quad \forall 1 \leq i \leq 5N/6.
\end{aligned}
\end{equation*}
The initial weights, however, are given by distinct block matrix structures. Let $J_1$ be an $N/6 \times N/6$ matrix.
The matrices $J_2$ and $J_3$ are $2 \times 2$ and $3 \times 3$ block matrices, respectively, where each block is the matrix $J_1$.
\begin{itemize}
\item The first system, $W^{(1)}(0)$, is defined by partitioning the $N$ agents into two equal-sized groups, giving a complete bipartite structure:
\begin{equation*}
W^{(1)}(0) = 2 \begin{pmatrix} 0 & J_{3} \\ J_{3} & 0 \end{pmatrix}.
\end{equation*}
\item The second system, $W^{(2)}(0)$, is defined by partitioning the agents into three equal-sized groups, forming a directed cycle:
\begin{equation*}
W^{(2)}(0) = 3 \begin{pmatrix} 0 & J_{2} & 0 \\ 0 & 0 & J_{2} \\ J_{2} & 0 & 0 \end{pmatrix}.
\end{equation*}
\end{itemize}
One can readily verify that for all $t \ge 0$, the state dynamics of the two systems are identical, i.e., $X^{(1)}_i(t) = X^{(2)}_i(t)$ for all $i \in [N]$, and 
\begin{equation*}
\begin{aligned}
X^{(1)}_i(t) = X^{(1)}_{i+N/6}(t), \quad \forall 1 \leq i \leq 5N/6.
\end{aligned}
\end{equation*}
This occurs because the total interaction experienced by any agent $i \in [N]$ is identical in both systems, albeit composed of contributions from different sets of agents. Nevertheless, despite this identical evolution, the initial weight configurations $W^{(1)}(0)$ and $W^{(2)}(0)$ can be far apart in the cut distance, yet in our $\delta_{W_1,\square}$. This observation suggests that our metric, while effective for proving convergence, may be too strict to capture the notion of dynamical similarity in all cases.
It turns out that the weaker $\gamma_{W_1,\square}$ topology mentioned in Section~\ref{subsec:coupling_and_tensorization} is suitable here, namely, by taking $\mathcal{Y}^{(k)} = ((X^{(k)}_{i})_{i \in [N]}, (W^{(k)}_{i,j})_{i,j \in [N]})$ on $\Omega = [N]$, $k = 1,2$, one has
\begin{equation*}
\begin{aligned}
\gamma_{W_1,\square}\big( \mathcal{Y}^{(1)}(t), \mathcal{Y}^{(2)}(t) \big) = 0, \quad \forall t \geq 0. 
\end{aligned}
\end{equation*}
The infimum $0$ is achieved by taking $\gamma \in \Pi([N], [N])$ as
\begin{equation*}
\begin{aligned}
\gamma_{i,j} = \frac{1}{6N} \mathbbm{1} \{ i - j \text{ mod } N/6 = 0 \}.
\end{aligned}
\end{equation*}
In other words, we ``mix the $6$ periods of $[N]$ together'' in the coupling. 

In addition, a intriguing observation arises from the long-term behavior of the weights. Since the state evolution is identical ($X^{(1)}(t) = X^{(2)}(t)$ for all $t$) and maintains the $N/6$ periodicity, it is easy to see from the solution form that
\begin{equation*}
W_{i,j}^{(k)}(t) - e^{-t/\epsilon} W^{(k)}_{i,j}(0) =  \int_0^t \epsilon^{-1} e^{-(t-s)/\epsilon} \beta(X_i(s), X_j(s)) \rd s
\end{equation*}
yields a matrix that is independent of $k \in \{1,2\}$ and has a $6 \times 6$ block structure. This heuristically suggests that the difference $W^{(1)}(t) - W^{(2)}(t)$ decays to zero as $t \to \infty$, implying convergence in $L^1$ and in the cut norm.

The intuition from this toy model is that for systems of the form \eqref{eqn:decay_SDEs}, some heterogeneous weight structures (which are identical only in $\gamma_{W_1,\square}$) may tend towards partial homogenization (which are identical in $\delta_{W_1,\square}$) over long time scales. We therefore formalize such intuition as the following expected result:
\begin{conj}[Partial homogenization]
Adopt the setting of Theorem~\ref{thm:empirical_data_stability}, but for $k \in \mathbb{N}$, assume the $N_k$-particle system $\bm{\mathcal{Y}}^{(k)}(t) = ((\bm{X}^{(k)}_{i}(t))_{i \in [N_k]}, (\bm{W}^{(k)}_{i,j}(t))_{i,j \in [N_k]})$ evolves according to the SDE system \eqref{eqn:decay_SDEs} with $\epsilon = 1$. 
Let $\overline{\mathcal{Y}}(t) = (\overline{\mathsf{X}}(t), \overline{\mathsf{W}}(t))$ denote the structural limit solution to the McKean-Vlasov SDEs \eqref{eqn:main_McKean_Vlasov} in the sense of Proposition~\ref{prop:well_posedness_McKean_Vlasov} (corresponding to the particular dynamics \eqref{eqn:decay_SDEs}).
Moreover, assume
\begin{equation*}
\begin{aligned}
& \lim_{k \to \infty} \mathbb{E} \Big[ \gamma_{W_1,\square}\big( \bm{\mathcal{Y}}^{(k)}(0), \overline{\mathcal{Y}}(0) \big) \Big] = 0.
\end{aligned}
\end{equation*}
Then, ``the convergence at $t \to \infty$ takes place in the stronger sense'', namely,
\begin{equation*}
\begin{aligned}
& \lim_{t \to \infty}
\limsup_{k \to \infty}
\E \Big[ \delta_{W_1,\square}\big( \bm{\mathcal{Y}}^{(k)}(t), \overline{\mathcal{Y}}(t) \big) \Big]= 0.
\end{aligned}
\end{equation*}
\end{conj}

\subsubsection{Measurement of self-organization} A final direction concerns the self-organization from homogeneous initial data, which is a key feature distinguishing the stochastic system \eqref{eqn:main_SDEs} ($\nu > 0$) from its deterministic counterpart. Consider a simple toy model where $X_i(0) \equiv 0 \in \mathbb{D}$ and $W_{i,j}(0) \equiv 1$ for all $i,j$. If $\nu = 0$, the system remains homogeneous for all time: $X_i(t)$ and $W_{i,j}(t)$ stay constant across all indices. In contrast, when $\nu > 0$, the individual Brownian motions first cause the states $\bm{X}_i(t)$ to diverge, which in turn induces heterogeneity in the weights $\bm{W}_{i,j}(t)$.

Our results guarantee that as $N \to \infty$, this system converges to a well-defined mean-field limit $\overline{\mathcal{Y}}(t) = (\overline{\mathsf{X}}(t), \overline{\mathsf{W}}(t))$ in the $\delta_{W_1,\square}$ metric. An initial question is whether this emergent limit is non-trivial, i.e., is the limiting kernel $\overline{\mathsf{W}}(t)$ non-constant? This is not immediately obvious, as a limit of pointwise divergent objects can still be constant (e.g., the convergence of an Erd\H{o}s-R\'enyi graph to a constant graphon).

A more profound question concerns the regularity of this self-organized limit $\overline{\mathcal{Y}}(t)$. We can possibly investigate this by analyzing the $N \to \infty$ convergence rate of the independent mixture sampling error $\mathbb{E} [ \delta_{W_1, \square}( \overline{\bm{\mathcal{Y}}}_N(t), \overline{\mathcal{Y}}(t) ) ]$ as in Lemma~\ref{lem:sampling_lemma},
or more simply, $\mathbb{E} [ \delta_{\square}( \overline{\bm{W}}_N(t), \overline{\mathsf{W}}(t))]$.
Lemma~\ref{lem:sampling_lemma} or the Sampling Lemma in graphon theory guarantees a baseline convergence rate of at least $O(1/\sqrt{\log N})$. However, if the limit possesses additional smoothness, for instance, $\overline{\mathsf{W}}(t)$ is a Lipschitz-bounded kernel on a low-dimensional manifold as its latent space $\Omega$, one would expect a faster, polynomial rate of convergence related to the concentration of empirical measures in the Wasserstein space \cite{boissard2011problemes, dobric1995asymptotics, fournier2015rate}, such as
\begin{equation*}
\begin{aligned}
\mathbb{E} [ \delta_{\square}( \overline{\bm{W}}_N(t), \overline{\mathsf{W}}(t))] \lesssim \mathbb{E}_{\bm{\xi}} \left[ W_1 \left( \frac{1}{N} \sum_{i \in [N]} \delta_{\bm{\xi}_i}, \pi \right) \right] \lesssim N^{-\theta}. 
\end{aligned}
\end{equation*}    
Therefore, encountering a sampling convergence rate better than the logarithmic baseline should strongly suggest the existence of a non-trivial emergent structure and a degree of regularity in the limit.

\section{Rigorous Setup} \label{sec:framework}

In this section, we provide the rigorous mathematical definitions and assumptions that underpin our main results.

\subsection{Notations and the McKean-Vlasov limit system} \label{subsec:limit_system}

Our first main objective is to establish the propagation of (asymptotic) dissociatedness (Theorem~\ref{thm:propagation_of_independent_mixture}). 
Recall from Proposition~\ref{prop:aldous_hoover} that dissociatedness in our context is characterized by an underlying state-weight structure $(\mathsf{X}, \mathsf{W})$ defined on a latent probability space $(\Omega, \mathcal{A}, \pi)$.

We define the McKean-Vlasov dynamics directly on such abstract representations, from which the finite system $(\overline{\bm{X}}(t), \overline{\bm{W}}(t))$ preserving dissociatedness in Theorem~\ref{thm:propagation_of_independent_mixture} can be naturally obtained via the sampling highlighted in Proposition~\ref{prop:aldous_hoover} and Lemma~\ref{lem:sampling_lemma} (for the sake of illustration, we assume here that the initial data are i.i.d., while noting that the general framework in Proposition~\ref{prop:aldous_hoover} allows for non-identical distributions to cover deterministic configurations):
\begin{itemize}
\item
The initial configuration of a dissociated $N$-system can be viewed as generated by sampling latent variables $\bm{\xi}_1, \dots, \bm{\xi}_N$ i.i.d.\ from $\Omega$.
While the Kolmogorov extension theorem guarantees the existence of a probability space supporting such a sequence, the most convenient choice to model it explicitly is the product space $\Omega^N$.
\item
The dynamics in \eqref{eqn:main_SDEs} are further driven by independent Brownian motions $\bm{B}_1, \dots, \bm{B}_N$.
Let $(\widehat{\Omega}, \widehat{\mathcal{F}}, \widehat{\pi})$ be the canonical path space for a single Brownian motion (equipped with the Wiener measure and the natural filtration).
With this setup, the evolution of the $N$-system is naturally modeled on the product space $(\Omega \times \widehat{\Omega})^N$.
It is therefore natural to characterize the propagation of dissociatedness by coupling the finite trajectories with i.i.d.\ samples from a representative process on $\Omega \times \widehat{\Omega}$, both realized on the common probability space $(\Omega \times \widehat{\Omega})^N$.

\item We equip the product space $\Omega \times \widehat{\Omega}$ with the filtration $(\mathcal{G}_t)_{t \ge 0}$ given by $\mathcal{G}_t \coloneqq \mathcal{A} \otimes \widehat{\mathcal{F}}_t$. 
This construction reflects the hybrid nature of the randomness: the structural label $\xi$ is static (fully known at $t=0$), while the path of the noise $\widehat{\xi}$ is revealed dynamically.
\end{itemize}

Accordingly, on this filtered probability space, we view the limit processes as representations taking values in the path space of continuous functions. 
Let $\mathcal{C}([0,T]; \mathbb{D})$ denote the space of continuous paths taking values in the state space $\mathbb{D}$. 
We define the state process $\overline{\mathsf{X}}$, the weight process $\overline{\mathsf{W}}$, and the canonical Brownian motion $\overline{\mathsf{B}}$ formally as the following maps:
\[
    \overline{\mathsf{X}}: \Omega \times \widehat{\Omega} \to \mathcal{C}([0,T]; \mathbb{D}), \quad
    \overline{\mathsf{W}}: (\Omega \times \widehat{\Omega})^2 \to \mathcal{C}([0,T]; \mathbb{R}), \quad
    \overline{\mathsf{B}}: \widehat{\Omega} \to \mathcal{C}([0,T]; \mathbb{R}^d).
\]
Here, $\overline{\mathsf{B}}$ is the canonical projection mapping a noise path $\widehat{\xi} \in \widehat{\Omega}$ to itself.
To maximize intuition and maintain notational alignment with \eqref{eqn:main_SDEs}, we place the latent variables in the subscripts: $\overline{\mathsf{X}}_{\xi, \widehat{\xi}}$ denotes the path $\overline{\mathsf{X}}(\xi, \widehat{\xi})$, with $\overline{\mathsf{X}}_{\xi, \widehat{\xi}}(t)$ representing its value at time $t$.
In the SDEs, we usually suppress the explicit time dependence $t$ to emphasize the path dynamics.

We consider the following McKean-Vlasov limit system on the filtered product space. 
Reflecting the fact that the initial configuration relies solely on the static structural labels, we initialize the system using a representation $(\overline{\mathsf{X}}_0, \overline{\mathsf{W}}_0)$ defined on $\Omega$ by setting
\[
   \overline{\mathsf{X}}_{\xi, \widehat{\xi}}(0) = \overline{\mathsf{X}}_0(\xi) \quad \text{and} \quad \overline{\mathsf{W}}_{(\xi, \widehat{\xi}), (\zeta, \widehat{\zeta})}(0) = \overline{\mathsf{W}}_0(\xi, \zeta).
\]
The evolution is governed by the following system: for $\pi \otimes \widehat{\pi}$-almost every realization $(\xi, \widehat{\xi}) \in \Omega \times \widehat{\Omega}$,
\begin{subequations}\label{eqn:main_McKean_Vlasov}
    \begin{equation} \label{eqn:main_McKean_Vlasov_X}
        \mathrm{d}\overline{\mathsf{X}}_{\xi, \widehat{\xi}} = \mu(\overline{\mathsf{X}}_{\xi, \widehat{\xi}}) \mathrm{d}t + \nu \mathrm{d}\overline{\mathsf{B}}_{\widehat{\xi}} + \mathbb{E}_{(\zeta, \widehat{\zeta}) \sim \pi \otimes \widehat{\pi}} \Big[ \overline{\mathsf{W}}_{(\xi, \widehat{\xi}), (\zeta, \widehat{\zeta})} \, \sigma(\overline{\mathsf{X}}_{\xi, \widehat{\xi}}, \overline{\mathsf{X}}_{\zeta, \widehat{\zeta}}) \Big] \mathrm{d}t,
    \end{equation}
    and for $(\pi \otimes \widehat{\pi})^{\otimes 2}$-almost every pair of realizations $((\xi, \widehat{\xi}), (\zeta, \widehat{\zeta})) \in (\Omega \times \widehat{\Omega})^2$,
    \begin{equation} \label{eqn:main_McKean_Vlasov_w}
        \mathrm{d}\overline{\mathsf{W}}_{(\xi, \widehat{\xi}), (\zeta, \widehat{\zeta})} = \alpha(\overline{\mathsf{X}}_{\xi, \widehat{\xi}}, \overline{\mathsf{X}}_{\zeta, \widehat{\zeta}}) \, \overline{\mathsf{W}}_{(\xi, \widehat{\xi}), (\zeta, \widehat{\zeta})} \mathrm{d}t + \beta(\overline{\mathsf{X}}_{\xi, \widehat{\xi}}, \overline{\mathsf{X}}_{\zeta, \widehat{\zeta}}) \mathrm{d}t.
    \end{equation}
\end{subequations}

We note that the system couples processes defined on differing probability spaces: the state process $\overline{\mathsf{X}}$ is determined by a single index $(\xi, \widehat{\xi})$, whereas the weight process $\overline{\mathsf{W}}$ is determined by a pair.
Accordingly, we constructively define the unique solution to \eqref{eqn:main_McKean_Vlasov} as the limit of a Picard iteration sequence based on stochastic integrals with respect to semimartingales, as detailed in Proposition~\ref{prop:well_posedness_McKean_Vlasov}.

\begin{rmk}[Representative Dynamics vs. Continuum of Agents]
It is important to clarify that our framework focuses on the representative dynamics, driven by a single canonical Brownian motion $\overline{\mathsf{B}}$.
This formulation circumvents the technical complexities associated with a continuum of independent Brownian motions.
A parallel line of research addresses such constructions via the Fubini extension framework and the notion of essentially pairwise independency (e.p.i.) (see \cite{sun2006exact} and recent applications in \cite{aurell2022stochasticII,coppini2025nonlinear,crucianelli2024interacting}).
In our context, however, the sampling perspective (Lemma~\ref{lem:sampling_lemma}) makes this unnecessary: it suffices to describe the evolution of the representative law on $\Omega \times \widehat{\Omega}$ to fully characterize the asymptotic behavior of the finite system.
\end{rmk}

\subsection{Regularity assumptions and well-posedness} \label{subsec:assumptions_wellposedness}

To ensure the well-posedness of the McKean-Vlasov system \eqref{eqn:main_McKean_Vlasov} and to facilitate the subsequent analysis, we impose the following regularity conditions on the coefficients and the initial data.

\begin{assu}[Regularity of Coefficients] \label{assu:coefficients}
    We assume that in the system \eqref{eqn:main_SDEs} or \eqref{eqn:main_McKean_Vlasov}, the coefficients $\mu, \sigma, \alpha, \beta$ are bounded and Lipschitz continuous with respect to their spatial arguments, and $\nu$ is a non-negative constant.

    In addition, we assume that the interaction kernel $\sigma$ and the weight update function $\alpha$ exhibit an ``integral separable structure''. Specifically, for $\sigma: \mathbb{D} \times \mathbb{D} \to \mathbb{R}^d$, this means there exists a measure space $E$ and functions $\sigma_1: \mathbb{D} \times E \to \mathbb{R}^d$ and $\sigma_2: \mathbb{D} \times E \to \mathbb{R}$ such that
    \[
        \sigma(x,y) = \int_{E} \sigma_1(x, \lambda) \, \sigma_2(y, \lambda) \, \mathrm{d}\lambda, \quad \text{for a.e. } (x,y) \in \mathbb{D} \times \mathbb{D},
    \]
    where the integral is absolutely convergent in the sense that
    \[
        \int_{E} \|\sigma_1(\cdot, \lambda)\|_{L^\infty} \, \|\sigma_2(\cdot, \lambda)\|_{L^\infty} \, \mathrm{d}\lambda < \infty.
    \]
    A similar structure is assumed for the function $\alpha: \mathbb{D} \times \mathbb{D} \to \mathbb{R}$.
\end{assu}

\begin{rmk}[On the Integral Separable Structure]
    The integral separable structure posited above is a technical condition designed to enable a factorization strategy within the norm estimation. 
    A prominent example of this structure is found in functions with absolutely integrable Fourier transforms (i.e., $\|\mathcal{F}\sigma\|_{L^1} < \infty$), as considered in \cite{ayi2024mean, bet2024weakly}. For such a function, the Fourier inversion formula naturally yields a separable representation:
    \begin{equation*}
        \sigma(x,y) = \int_{\mathbb{R}^{2d}} (\mathcal{F} \sigma)(\lambda_1,\lambda_2) \, e^{2 \pi i x \cdot \lambda_1} \, e^{2 \pi i y \cdot \lambda_2} \, \mathrm{d}(\lambda_1, \lambda_2).
    \end{equation*}
    (Note that a standard decomposition of the complex exponentials ensures the factors $\sigma_1, \sigma_2$ remain real-valued.)
    While the Fourier transform provides a practical criterion for sufficiently smooth kernels (typically resulting fast Fourier decay) to satisfy Assumption~\ref{assu:coefficients}, it does not encompass all cases, as there exist many bounded ($L^\infty$) functions whose Fourier transforms are not in $L^1$.
\end{rmk}

In addition to the regularity of the coefficients, we require specific moment bounds. While these conditions are defined below for general random variables, in the context of our well-posedness and convergence analysis, they are primarily imposed on the initial configurations of the system (i.e., at time $t=0$).

\begin{assu}[Moment Bounds on Data] \label{assu:initial_data}
    When we state that a system satisfies Assumption~\ref{assu:initial_data}, we refer to the following conditions depending on the underlying structure:
    \begin{itemize}
        \item For a general representation given by measurable maps $\mathsf{X}: \Omega \to \mathbb{D}$ and $\mathsf{W}: \Omega \times \Omega \to \mathbb{R}$, we require that for some constant $C > 0$,
        \begin{equation*}
            \|\mathsf{X}\|_{L^2(\Omega)} \coloneqq \left( \int_{\Omega} |\mathsf{X}(\xi)|^2 \, \mathrm{d}\pi(\xi) \right)^{1/2} < C, \quad \text{and} \quad \|\mathsf{W}\|_{L^\infty(\Omega^2)} < C.
        \end{equation*}

        \item For a finite collection of random variables $((\bm{X}_{i})_{i \in [N]}, (\bm{W}_{i,j})_{i,j \in [N]})$, which are not necessarily dissociated, we require that for some constant $C > 0$,
        \begin{equation*}
            \frac{1}{N} \sum_{i \in [N]} \mathbb{E} \big[ |\bm{X}_i|^2 \big] < C, \quad \text{and} \quad \max_{i,j \in [N]} |\bm{W}_{i,j}| < C \quad \text{almost surely}.
        \end{equation*}

        \item For a sequence of systems (indexed by $k$ or $N$), the preceding conditions are assumed to hold uniformly with respect to the index.
    \end{itemize}
\end{assu}

With these assumptions, the construction of the solution via Picard iteration is mathematically justified.

\begin{prop}[Well-posedness and Stability of McKean-Vlasov SDEs] \label{prop:well_posedness_McKean_Vlasov}
    Consider the SDE system \eqref{eqn:main_McKean_Vlasov} with coefficients satisfying Assumption~\ref{assu:coefficients}.
    Assume that the initial data $(\overline{\mathsf{X}}_0, \overline{\mathsf{W}}_0)$ defined on the probability space $(\Omega, \mathcal{A}, \pi)$ satisfies Assumption~\ref{assu:initial_data}.
    Let $(\overline{\mathsf{B}}(t))_{t \in [0,T]}$ be the canonical Brownian motion defined on the filtered probability space $(\widehat{\Omega}, \widehat{\mathcal{F}}, (\widehat{\mathcal{F}}_t)_{t \ge 0}, \widehat{\pi})$.

    Then, there exists a unique (in the almost sure sense) pair of stochastic processes
    \begin{equation*}
        (\overline{\mathsf{X}}_{\xi,\widehat{\xi}}(t))_{(\xi,\widehat{\xi}) \in \Omega \times \widehat{\Omega}, t \in [0,T]} \quad \text{and} \quad (\overline{\mathsf{W}}_{(\xi,\widehat{\xi}),(\zeta,\widehat{\zeta})}(t))_{(\xi,\widehat{\xi}),(\zeta,\widehat{\zeta}) \in (\Omega \times \widehat{\Omega})^2, t \in [0,T]}
    \end{equation*}
    satisfying the integral equations corresponding to \eqref{eqn:main_McKean_Vlasov} almost surely on $[0,T]$, with the following properties:
    \begin{itemize}
        \item Adaptedness: The process $\overline{\mathsf{X}}(t)$ is measurable with respect to $\mathcal{A} \otimes \widehat{\mathcal{F}}_t$, and the process $\overline{\mathsf{W}}(t)$ is measurable with respect to the product filtration $(\mathcal{A} \otimes \widehat{\mathcal{F}}_t)^{\otimes 2}$.

        \item Moment Control: For some constant $C_T > 0$ depending on $T$ and the bounds in Assumptions~\ref{assu:coefficients} and \ref{assu:initial_data}, the following holds for all $t \in [0,T]$:
        \begin{equation*}
            \|\overline{\mathsf{X}}(t)\|_{L^2(\Omega \times \widehat{\Omega})} < C_T \quad \text{and} \quad \|\overline{\mathsf{W}}(t)\|_{L^\infty((\Omega \times \widehat{\Omega})^2)} < C_T.
        \end{equation*}

        \item Path Regularity: The solutions have continuous sample paths almost surely.
    \end{itemize}

Furthermore, consider two sets of initial data $(\overline{\mathsf{X}}^{(k)}_0, \overline{\mathsf{W}}^{(k)}_0)$ ($k=1,2$) satisfying Assumption~\ref{assu:initial_data}.
If they are defined on the same latent space $\Omega$, then for any $t \in [0,T]$, the following estimate holds:
\begin{equation*}
\begin{aligned}
&\|\overline{\mathsf{X}}^{(1)}(t) - \overline{\mathsf{X}}^{(2)}(t)\|_{L^1(\Omega \times \widehat{\Omega})} + \|\overline{\mathsf{W}}^{(1)}(t) - \overline{\mathsf{W}}^{(2)}(t)\|_{\square(\Omega \times \widehat{\Omega})}
\\
& \quad \leq C(t) \big( \|\overline{\mathsf{X}}^{(1)}_0 - \overline{\mathsf{X}}^{(2)}_0\|_{L^1(\Omega)} + \|\overline{\mathsf{W}}^{(1)}_0 - \overline{\mathsf{W}}^{(2)}_0\|_{\square(\Omega)} \big).
\end{aligned}
\end{equation*}
Here, $C(t) > 0$ is a constant depending on time, dimension, and the bounds in Assumptions~\ref{assu:coefficients} and \ref{assu:initial_data}.
For general initial data (not necessarily defined on the same space), we have the stability in the unlabeled distance:
\begin{equation*}
\begin{aligned}
&\delta_{W_1,\square} \big( ( \overline{\mathsf{X}}^{(1)}(t), \overline{\mathsf{W}}^{(1)}(t) ) , ( \overline{\mathsf{X}}^{(2)}(t), \overline{\mathsf{W}}^{(2)}(t) ) \big)
 \leq C(t) \delta_{W_1,\square} \big( ( \overline{\mathsf{X}}^{(1)}_0, \overline{\mathsf{W}}^{(1)}_0 ) , ( \overline{\mathsf{X}}^{(2)}_0, \overline{\mathsf{W}}^{(2)}_0 ) \big).
\end{aligned}
\end{equation*}

\end{prop}

\begin{rmk}[Link to Finite Dissociated Systems] \label{rmk:finite_sampling}
The finite dissociated system $(\overline{\bm{X}}(t), \overline{\bm{W}}(t))$ in Theorem~\ref{thm:propagation_of_independent_mixture}, starting with initial data $(\bm{X}_{(0)}, \bm{W}_{(0)})$, is constructed as follows.
By Proposition~\ref{prop:aldous_hoover}, this initial data is represented by $(\overline{\mathsf{X}}_0, \overline{\mathsf{W}}_0)$ and independent variables $\bm{\xi}_1, \dots, \bm{\xi}_N$ on a latent space $\Omega$, where the distributions $\pi_i = \Law(\bm{\xi}_i)$ satisfy the averaging condition $\pi = \frac{1}{N} \sum_{i=1}^N \pi_i$.
Given this initial structure $(\overline{\mathsf{X}}_0, \overline{\mathsf{W}}_0)$ and the measure $\pi$, Proposition~\ref{prop:well_posedness_McKean_Vlasov} yields $(\overline{\mathsf{X}}, \overline{\mathsf{W}})$.
We let $\widehat{\bm{\xi}}_i$ be the canonical path corresponding to the \emph{same} Brownian motion $\bm{B}_i$ that drives the original particle system \eqref{eqn:main_SDEs}.
The finite system is then defined via the sampling:
\begin{equation} \label{eqn:sampling_definition}
\overline{\bm{X}}_i(t) \coloneqq \overline{\mathsf{X}}_{\bm{\xi}_i, \widehat{\bm{\xi}}_i}(t), \qquad \overline{\bm{W}}_{i,j}(t) \coloneqq \overline{\mathsf{W}}_{(\bm{\xi}_i, \widehat{\bm{\xi}}_i), (\bm{\xi}_j, \widehat{\bm{\xi}}_j)}(t).
\end{equation}
Crucially, sharing the same Brownian motions $\bm{B}_i$ between the original system \eqref{eqn:main_SDEs} and these auxiliary processes \eqref{eqn:sampling_definition} constitutes the coupling on the same probability space
\begin{equation*}
\begin{aligned}
\Big( (\Omega \times \widehat{\Omega})^N, (\mathcal{A} \otimes \widehat{\mathcal{F}})^{\otimes N}, \bigotimes_{i=1}^N (\pi_i \otimes \widehat{\pi} ) \Big),
\end{aligned}
\end{equation*}
which is required for the pathwise convergence analysis.
\end{rmk}

\subsection{The \texorpdfstring{$\delta_{W_1, \square}$}{delta W1 box} distance} \label{subsec:metric_def}

To establish a rigorous notion for our quantitative estimates, we first consider the comparison of two state-weight structures $\mathsf{X}^{(k)}: \Omega \to \mathbb{D}$ and $\mathsf{W}^{(k)}: \Omega \times \Omega \to \mathbb{R}$, $k=1,2$, defined on the \textit{same} probability space $(\Omega, \mathcal{A}, \pi)$. In this setting, the discrepancy can be naturally measured via the hybrid norm:
\begin{equation*}
    \|\mathsf{X}^{(1)} - \mathsf{X}^{(2)}\|_{L^1(\Omega)} + \|\mathsf{W}^{(1)} - \mathsf{W}^{(2)}\|_{\square}.
\end{equation*}
Recall that $\|\cdot\|_{\square}$ denotes the cut norm in the theory of graph limits (see \cite{lovasz2012large}), defined for a kernel $\mathsf{W}: \Omega \times \Omega \to \mathbb{R}$ as
\begin{equation*}
    \|\mathsf{W}\|_{\square} \coloneqq \sup_{S,T \subseteq \Omega} \left| \int_{S \times T} \mathsf{W}(\xi, \zeta) \, \mathrm{d}\pi(\xi) \, \mathrm{d}\pi(\zeta) \right|,
\end{equation*}
where the supremum is taken over all measurable subsets $S, T$ of $\Omega$. To explicitly emphasize the underlying probability space, we will occasionally adopt the notation $\|\cdot\|_{\square(\Omega)}$.

To extend this comparison to systems defined on differing probability spaces, we employ measure-preserving maps. This allows us to define the ``unlabeled'' distance that captures the intrinsic structural similarity between systems, independent of the specific latent spaces.

\begin{defi}[Measure-preserving Map] \label{def:mp_map}
    For two probability spaces $(\Omega, \mathcal{A}, \pi)$ and $(\Omega', \mathcal{A}', \pi')$, a measurable map $\varphi: \Omega \to \Omega'$ is called \textit{measure-preserving} if
    \begin{equation*}
        \pi(\varphi^{-1}(A')) = \pi'(A') \quad \text{for all } A' \in \mathcal{A}'.
    \end{equation*}
    That is, the push-forward of $\pi$ by $\varphi$ coincides with $\pi'$, denoted as $\varphi_{\#} \pi = \pi'$.
\end{defi}

For a system $(\mathsf{X}, \mathsf{W})$ defined on the target space $\Omega'$ and a measure-preserving map $\varphi: \Omega \to \Omega'$, we define the \textit{pullback} (or relabeling) $(\mathsf{X}_\varphi, \mathsf{W}_\varphi)$ on the domain $\Omega$ as:
\begin{equation*}
    \mathsf{X}_\varphi(\cdot) = \mathsf{X} (\varphi(\cdot) ) \quad \text{and} \quad \mathsf{W}_\varphi (\cdot, \cdot) = \mathsf{W} (\varphi(\cdot), \varphi(\cdot)).
\end{equation*}
The unlabeled distance $\delta_{W_1, \square}$ is then defined by minimizing the hybrid norm difference between two such pullbacks on a common domain.

\begin{defi}[The distance $\delta_{W_1,\square}$] \label{defi:metric_delta}
    The distance between $(\mathsf{X}^{(1)}, \mathsf{W}^{(1)})$ and $(\mathsf{X}^{(2)}, \mathsf{W}^{(2)})$ defined on $\Omega^{(1)}$ and $\Omega^{(2)}$ respectively is defined as:
    \begin{equation*}
        \delta_{W_1,\square}\Big( (\mathsf{X}^{(1)}, \mathsf{W}^{(1)}), (\mathsf{X}^{(2)}, \mathsf{W}^{(2)}) \Big)
        \coloneqq
        \inf_{\varphi_1, \varphi_2} \left( \|\mathsf{X}^{(1)}_{\varphi_1} - \mathsf{X}^{(2)}_{\varphi_2}\|_{L^1(\Omega)} + \|\mathsf{W}^{(1)}_{\varphi_1} - \mathsf{W}^{(2)}_{\varphi_2}\|_{\square(\Omega)} \right),
    \end{equation*}
    where the infimum ranges over all probability spaces $(\Omega, \mathcal{A}, \pi)$ and measure-preserving maps $\varphi_k: \Omega \to \Omega^{(k)}$ ($k=1,2$).
\end{defi}

This metric is a natural generalization of existing concepts:
\begin{itemize}
    \item If the weights are trivial (e.g., $\mathsf{W}^{(k)} \equiv 1$), the interaction term vanishes, and the definition reduces to the classical Wasserstein-1 distance between the laws of $\mathsf{X}^{(1)}$ and $\mathsf{X}^{(2)}$:
    \begin{equation*}
        \inf_{\varphi_1, \varphi_2} \|\mathsf{X}^{(1)}_{\varphi_1} - \mathsf{X}^{(2)}_{\varphi_2}\|_{L^1} = W_1(\mathrm{Law}(\mathsf{X}^{(1)}), \mathrm{Law}(\mathsf{X}^{(2)})).
    \end{equation*}
    \item Conversely, omitting the state term $\mathsf{X}$ recovers the graphon cut distance $\delta_{\square}(\mathsf{W}^{(1)}, \mathsf{W}^{(2)})$ as defined in graph limit theory (see \cite[Chapters 8, 13]{lovasz2012large}).
\end{itemize}

\section{Proofs of Main Results} \label{sec:proofs}

This section contains the detailed proofs of the results stated in Section~\ref{sec:overview} and the well-posedness of the limit system defined in Section~\ref{sec:framework}.

\subsection{Proof of well-posedness (Proposition~\ref{prop:well_posedness_McKean_Vlasov})} \label{subsec:proof_well_posedness}

We construct the solution as a fixed point of an integral operator acting on the space of stochastic processes.
Let $\overline{\mathcal{Y}} = (\overline{\mathsf{X}}, \overline{\mathsf{W}})$ denote a pair of $(\mathcal{A} \otimes \widehat{\mathcal{F}}_t)$-adapted and $(\mathcal{A} \otimes \widehat{\mathcal{F}}_t)^{\otimes 2}$-adapted processes.
We define the integral operator $\Psi(\overline{\mathcal{Y}}) = (\Psi^{\mathsf{X}}(\overline{\mathcal{Y}}), \Psi^{\mathsf{W}}(\overline{\mathcal{Y}}))$ component-wise as follows:
for $(\pi \otimes \widehat{\pi})$-a.e. $(\xi, \widehat{\xi}), (\zeta, \widehat{\zeta})$ and $t \in [0,T]$,
\begin{equation} \label{eqn:operator_def}
\begin{aligned}
&\hspace{1cm} \begin{aligned}
\Psi^{\mathsf{X}}(\overline{\mathcal{Y}})_{\xi, \widehat{\xi}}(t) \coloneqq \overline{\mathsf{X}}_0(\xi) & + \int_0^t \mu(\overline{\mathsf{X}}_{\xi, \widehat{\xi}}(s)) \, \mathrm{d}s + \nu \overline{\mathsf{B}}_{\widehat{\xi}}(t)
\\
& + \int_0^t \mathbb{E}_{(\zeta, \widehat{\zeta}) \sim (\pi \otimes \widehat{\pi})} \Big[ \overline{\mathsf{W}}_{(\xi, \widehat{\xi}), (\zeta, \widehat{\zeta})}(s) \, \sigma(\overline{\mathsf{X}}_{\xi, \widehat{\xi}}(s), \overline{\mathsf{X}}_{\zeta, \widehat{\zeta}}(s)) \Big] \, \mathrm{d}s,
\end{aligned}
\\
&\begin{aligned}
\Psi^{\mathsf{W}}(\overline{\mathcal{Y}})_{(\xi, \widehat{\xi}), (\zeta, \widehat{\zeta})}(t) \coloneqq \overline{\mathsf{W}}_0(\xi, \zeta) & + \int_0^t \alpha(\overline{\mathsf{X}}_{\xi, \widehat{\xi}}(s), \overline{\mathsf{X}}_{\zeta, \widehat{\zeta}}(s)) \, \overline{\mathsf{W}}_{(\xi, \widehat{\xi}), (\zeta, \widehat{\zeta})}(s) \, \mathrm{d}s
\\
& + \int_0^t \beta(\overline{\mathsf{X}}_{\xi, \widehat{\xi}}(s), \overline{\mathsf{X}}_{\zeta, \widehat{\zeta}}(s)) \, \mathrm{d}s.
\end{aligned}
\end{aligned}
\end{equation}
The well-posedness and stability rely on the following key estimates for $\Psi$.

\begin{lem}[A Priori Estimates] \label{lem:picard_estimates}
Let $\overline{\mathcal{Y}} = (\overline{\mathsf{X}}, \overline{\mathsf{W}})$ be an adapted process.
The operator $\Psi$ satisfies the following moment propagation estimate: for any $t \in [0,T]$,
\begin{equation*}
\begin{aligned}
&\|\Psi^{\mathsf{X}}(\overline{\mathcal{Y}})\|_{t, L^2}^2 \leq 4 \bigg( \|\overline{\mathsf{X}}_0\|_{L^2}^2 + \|\mu\|_{L^\infty}^2 t^2 + d \nu^2 t + \|\overline{\mathsf{W}}\|_{t, L^\infty}^2 \|\sigma\|_{L^\infty}^2 t^2 \bigg),
\\
&\|\Psi^{\mathsf{W}}(\overline{\mathcal{Y}})\|_{t, L^\infty} \leq \|\overline{\mathsf{W}}_0\|_{L^\infty} + \int_0^t \Big( \|\alpha\|_{L^\infty} \|\overline{\mathsf{W}}\|_{s, L^\infty} + \|\beta\|_{L^\infty} \Big) \, \mathrm{d}s.
\end{aligned}
\end{equation*}
Here, the norms are defined as 
\begin{equation*}
\begin{aligned}
\|\overline{\mathsf{W}}\|_{t, L^\infty} \coloneqq \sup_{s \in [0,t]} \|\overline{\mathsf{W}}(s)\|_{L^\infty((\Omega \times \widehat{\Omega})^2)},
\quad \|\overline{\mathsf{X}}\|_{t, L^2} \coloneqq \sup_{s \in [0,t]} \|\overline{\mathsf{X}}(s)\|_{L^2(\Omega \times \widehat{\Omega})}.
\end{aligned}
\end{equation*}

Moreover, $\Psi$ satisfies the Lipschitz continuity estimate.
Let $\overline{\mathcal{Y}}^{(k)} = (\overline{\mathsf{X}}^{(k)}, \overline{\mathsf{W}}^{(k)})$ ($k=1,2$) be two adapted processes on the same latent space $\Omega$.
Then, for any $t \in [0,T]$,
\begin{equation*}
\begin{aligned}
& \|\Psi^{\mathsf{X}}(\overline{\mathcal{Y}}^{(1)}) - \Psi^{\mathsf{X}}(\overline{\mathcal{Y}}^{(2)})\|_{t, L^1} + \|\Psi^{\mathsf{W}}(\overline{\mathcal{Y}}^{(1)}) - \Psi^{\mathsf{W}}(\overline{\mathcal{Y}}^{(2)})\|_{t, \square}
\\
& \quad \leq \|\overline{\mathsf{X}}^{(1)}_0 - \overline{\mathsf{X}}^{(2)}_0\|_{L^1} + \|\overline{\mathsf{W}}^{(1)}_0 - \overline{\mathsf{W}}^{(2)}_0\|_{\square}
+ C(t) \int_0^t \Big( \|\overline{\mathsf{X}}^{(1)} - \overline{\mathsf{X}}^{(2)}\|_{s, L^1} + \|\overline{\mathsf{W}}^{(1)} - \overline{\mathsf{W}}^{(2)}\|_{s, \square} \Big) \, \mathrm{d}s,
\end{aligned}
\end{equation*}
where $C(t)$
depends on time, dimension, and the bounds in Assumptions~\ref{assu:coefficients} and $\|\overline{\mathsf{W}}_0\|_{L^\infty}$.
This estimate remains valid if all the cut norms $\|\cdot\|_{\square}$ are replaced by $L^1$ norms.
\end{lem}

\begin{proof}[Proof of Lemma~\ref{lem:picard_estimates}]
The moment estimates follow directly from Hölder's inequalities
and the Ito isometry $\mathbb{E}[|\overline{\mathsf{B}}(t)|^2] = d t$ ($d$ for dimension); we omit the detailed derivation.

We now turn to the Lipschitz continuity.
For brevity, let $z \coloneqq (\xi, \widehat{\xi})$ and $z' \coloneqq (\zeta, \widehat{\zeta})$ denote generic elements in the product space $\Omega \times \widehat{\Omega}$.
In the estimates below, we implicitly assume that integrals $\int \dots \mathrm{d}z$ and expectations $\mathbb{E}_{z'}$ are taken with respect to the default product measure $\pi \otimes \widehat{\pi}$.

\textit{Step 1: Estimate for the state component $\Psi^{\mathsf{X}}$.}
From the definition in \eqref{eqn:operator_def}, the $L^1$-difference of the state components satisfies:
\begin{equation*}
\begin{aligned}
& \|\Psi^{\mathsf{X}}(\overline{\mathcal{Y}}^{(1)}) - \Psi^{\mathsf{X}}(\overline{\mathcal{Y}}^{(2)})\|_{t, L^1}
\\
& \quad \leq \|\overline{\mathsf{X}}^{(1)}_0 - \overline{\mathsf{X}}^{(2)}_0\|_{L^1} + \int_0^t \int_{\Omega \times \widehat{\Omega}} \left| \mu(\overline{\mathsf{X}}^{(1)}_z(s)) - \mu(\overline{\mathsf{X}}^{(2)}_z(s)) \right| \, \mathrm{d}z \, \mathrm{d}s
\\
& \qquad + \int_0^t \underbrace{ \int_{\Omega \times \widehat{\Omega}} \left| \mathbb{E}_{z'} \Big[ \overline{\mathsf{W}}^{(1)}_{z,z'}(s) \sigma(\overline{\mathsf{X}}^{(1)}_z(s), \overline{\mathsf{X}}^{(1)}_{z'}(s)) - \overline{\mathsf{W}}^{(2)}_{z,z'}(s) \sigma(\overline{\mathsf{X}}^{(2)}_z(s), \overline{\mathsf{X}}^{(2)}_{z'}(s)) \Big] \right| \, \mathrm{d}z }_{\eqqcolon \mathcal{I}(s)} \, \mathrm{d}s.
\end{aligned}
\end{equation*}
The Brownian terms are identical for both processes and cancel out.
Since $\mu$ is Lipschitz, the drift term is bounded by $\|\nabla \mu\|_{L^\infty} \int_0^t \|\overline{\mathsf{X}}^{(1)} - \overline{\mathsf{X}}^{(2)}\|_{s, L^1} \, \mathrm{d}s$.
For the interaction term $\mathcal{I}(s)$, we add and subtract a cross term to decompose the integrand into two parts, and write the expectation $\mathbb{E}_{z'}$ as an integral $\int \dots \mathrm{d}z'$:
\begin{equation*}
\begin{aligned}
\mathcal{I}(s) &= \int_{\Omega \times \widehat{\Omega}} \bigg| \int_{\Omega \times \widehat{\Omega}} \Big( \Delta_1(s, z, z') + \Delta_2(s, z, z') \Big) \, \mathrm{d}z' \bigg| \, \mathrm{d}z,
\end{aligned}
\end{equation*}
where
\begin{equation*}
\begin{aligned}
\Delta_1 = \overline{\mathsf{W}}^{(1)}_{z,z'} \Big( \sigma(\overline{\mathsf{X}}^{(1)}_z, \overline{\mathsf{X}}^{(1)}_{z'}) - \sigma(\overline{\mathsf{X}}^{(2)}_z, \overline{\mathsf{X}}^{(2)}_{z'}) \Big),
\quad
\Delta_2 = \Big( \overline{\mathsf{W}}^{(1)}_{z,z'} - \overline{\mathsf{W}}^{(2)}_{z,z'} \Big) \sigma(\overline{\mathsf{X}}^{(2)}_z, \overline{\mathsf{X}}^{(2)}_{z'}).
\end{aligned}
\end{equation*}
For $\Delta_1$, using $\|\overline{\mathsf{W}}^{(1)}\|_{t, L^\infty} \leq C$ (depending on time, Assumptions~\ref{assu:coefficients} and $\|\overline{\mathsf{W}}_0\|_{L^\infty}$ from the first part a priori moment estimate) and the Lipschitz continuity of $\sigma$, we obtain:
\begin{equation*}
\begin{aligned}
\int \left| \int \Delta_1 \, \mathrm{d}z' \right| \, \mathrm{d}z 
&\leq \|\overline{\mathsf{W}}^{(1)}\|_{t, L^\infty} \|\nabla \sigma\|_{L^\infty} \int_{\Omega \times \widehat{\Omega}} \int_{\Omega \times \widehat{\Omega}} \Big( |\overline{\mathsf{X}}^{(1)}_z - \overline{\mathsf{X}}^{(2)}_z| + |\overline{\mathsf{X}}^{(1)}_{z'} - \overline{\mathsf{X}}^{(2)}_{z'}| \Big) \, \mathrm{d}z' \, \mathrm{d}z
\\
&= 2 \|\overline{\mathsf{W}}^{(1)}\|_{t, L^\infty} \|\nabla \sigma\|_{L^\infty} \|\overline{\mathsf{X}}^{(1)} - \overline{\mathsf{X}}^{(2)}\|_{s, L^1}.
\end{aligned}
\end{equation*}
For $\Delta_2$, we exploit the integral separable structure of $\sigma$ (Assumption~\ref{assu:coefficients}) to isolate the cut norm.
Substituting $\sigma(x,y) = \int_E \sigma_1(x, \lambda) \sigma_2(y, \lambda) \, \mathrm{d}\lambda$, the term becomes:
\begin{equation*}
\begin{aligned}
\int_{\Omega \times \widehat{\Omega}} \left| \int_{\Omega \times \widehat{\Omega}} \Delta_2 \, \mathrm{d}z' \right| \, \mathrm{d}z
= \int_{\Omega \times \widehat{\Omega}} \bigg| \int_E \sigma_1(\overline{\mathsf{X}}^{(2)}_z, \lambda) \bigg( \int_{\Omega \times \widehat{\Omega}} \big( \overline{\mathsf{W}}^{(1)}_{z,z'} - \overline{\mathsf{W}}^{(2)}_{z,z'} \big) \sigma_2(\overline{\mathsf{X}}^{(2)}_{z'}, \lambda) \, \mathrm{d}z' \bigg) \, \mathrm{d}\lambda \bigg| \, \mathrm{d}z.
\end{aligned}
\end{equation*}
Considering the direction of the inequality, we move the integral over $E$ outside the absolute value.
Let $K(z,z') \coloneqq \overline{\mathsf{W}}^{(1)}_{z,z'}(s) - \overline{\mathsf{W}}^{(2)}_{z,z'}(s)$ and $f_\lambda(z') \coloneqq \sigma_2(\overline{\mathsf{X}}^{(2)}_{z'}(s), \lambda)$.
The inner integral over $z'$ can be recognized as the integral operator $T_K$ acting on the function $f_\lambda$.
Specifically,
\begin{equation*}
\begin{aligned}
\int_{\Omega \times \widehat{\Omega}} \bigg| \int_{\Omega \times \widehat{\Omega}} K(z,z') f_\lambda(z') \, \mathrm{d}z' \bigg| \, \mathrm{d}z 
&= \| T_K(f_\lambda) \|_{L^1} \leq \| T_K \|_{L^\infty \to L^1} \|f_\lambda\|_{L^\infty}.
\end{aligned}
\end{equation*}
Using the standard cut norm inequality $\|T_K\|_{L^\infty \to L^1} \leq 4 \|K\|_{\square}$ (see \cite[Lemma 8.11]{lovasz2012large}), we have
\begin{equation*}
\| T_K(f_\lambda) \|_{L^1} \leq 4 \|\overline{\mathsf{W}}^{(1)}(s) - \overline{\mathsf{W}}^{(2)}(s)\|_{\square(\Omega \times \widehat{\Omega})} \, \|\sigma_2(\cdot, \lambda)\|_{L^\infty}.
\end{equation*}
Finally, integrating over $\lambda \in E$ yields the bound for the $\Delta_2$ term:
\begin{equation*}
\int_{\Omega \times \widehat{\Omega}} \left| \int_{\Omega \times \widehat{\Omega}} \Delta_2 \, \mathrm{d}z' \right| \, \mathrm{d}z \leq 4 \left( \int_E \|\sigma_1(\cdot, \lambda)\|_{L^\infty} \|\sigma_2(\cdot, \lambda)\|_{L^\infty} \, \mathrm{d}\lambda \right) \|\overline{\mathsf{W}}^{(1)} - \overline{\mathsf{W}}^{(2)}\|_{s, \square}.
\end{equation*}

\textit{Step 2: Estimate for the weight component $\Psi^{\mathsf{W}}$.}
We can apply the same separation of variables technique for $\alpha$.
We rely on the following property of the cut norm: for functions $u = u(z)$, $w = w(z,z')$, and $v = v(z')$, we have
\begin{equation*}
\|u w v\|_{\square} \leq 4 \|u\|_{L^\infty} \|v\|_{L^\infty} \|w\|_{\square}.
\end{equation*}
Applying this property to handle the $\alpha \overline{\mathsf{W}}$ product, and bounding the cut norm by the $L^1$ norm for the Lipschitz differences of $\alpha$ and $\beta$, we directly obtain the following estimate:
\begin{equation*}
\begin{aligned}
& \|\Psi^{\mathsf{W}}(\overline{\mathcal{Y}}^{(1)}) - \Psi^{\mathsf{W}}(\overline{\mathcal{Y}}^{(2)})\|_{t, \square}
\\
& \quad \leq \|\overline{\mathsf{W}}^{(1)}_0 - \overline{\mathsf{W}}^{(2)}_0\|_{\square}
+ \int_0^t 2 \Big( \|\nabla \beta\|_{L^\infty} + \|\nabla \alpha\|_{L^\infty} \|\overline{\mathsf{W}}^{(2)}\|_{t, L^\infty} \Big) \|\overline{\mathsf{X}}^{(1)} - \overline{\mathsf{X}}^{(2)}\|_{s, L^1} \, \mathrm{d}s
\\
& \qquad + 4 \left( \int_E \|\alpha_1(\cdot, \lambda)\|_{L^\infty} \|\alpha_2(\cdot, \lambda)\|_{L^\infty} \, \mathrm{d}\lambda \right) \int_0^t \|\overline{\mathsf{W}}^{(1)} - \overline{\mathsf{W}}^{(2)}\|_{s, \square} \, \mathrm{d}s.
\end{aligned}
\end{equation*}
Combining the estimates from Step 1 and Step 2 yields the desired Lipschitz inequality.

\end{proof}

With the key estimates established, we complete the proof of the main proposition.

\begin{proof}[Proof of Proposition~\ref{prop:well_posedness_McKean_Vlasov}]

\textit{Existence and Uniqueness.}
Using the moment propagation estimates from Lemma~\ref{lem:picard_estimates}, we can construct invariant sets for the integral operator $\Psi$. 
For initial data $(\overline{\mathsf{X}}_0, \overline{\mathsf{W}}_0)$ with bounded moments, consider the set $\mathcal{M}$ consisting of pairs $(\overline{\mathsf{X}}, \overline{\mathsf{W}})$ satisfying the uniform bounds for all $t \in [0,T]$:
\begin{align*}
& \|\overline{\mathsf{W}}(t)\|_{L^\infty} \leq e^{\|\alpha\|_{L^\infty}t} \left( \|\overline{\mathsf{W}}_0\|_{L^\infty} + \|\beta\|_{L^\infty}t \right), \\
& \|\overline{\mathsf{X}}(t)\|_{L^2}^2 \leq 4 \bigg( \|\overline{\mathsf{X}}_0\|_{L^2}^2 + \|\mu\|_{L^\infty}^2 t^2 + d \nu^2 t + \left[ e^{\|\alpha\|_{L^\infty}t} \left( \|\overline{\mathsf{W}}_0\|_{L^\infty} + \|\beta\|_{L^\infty}t \right) \right]^2 \|\sigma\|_{L^\infty}^2 t^2 \bigg).
\end{align*}
The invariance of $\mathcal{M}$ under $\Psi$ follows directly from the moment estimates of Lemma~\ref{lem:picard_estimates}. 
Within this controlled regime, we implement a standard Picard iteration scheme. 
From the Lipschitz estimate of Lemma~\ref{lem:picard_estimates}, the operator $\Psi$ is a contraction on $[0, \delta]$ for sufficiently small $\delta > 0$. 
We then iteratively extend the solution to the full interval $[0, T]$, ensuring global existence and uniqueness while maintaining the uniform bounds provided by $\mathcal{M}$.

\textit{Stability.}
Let $\overline{\mathcal{Y}}^{(k)} = (\overline{\mathsf{X}}^{(k)}, \overline{\mathsf{W}}^{(k)})$ ($k=1,2$) be solutions starting with initial data $\overline{\mathcal{Y}}_0^{(k)}$ respectively.
Since they are fixed points of $\Psi$, we apply the Lipschitz estimate of Lemma~\ref{lem:picard_estimates} to obtain:
\begin{equation*}
\Delta(t) \leq \Delta(0) + C(t) \int_0^t \Delta(s) \, \mathrm{d}s,
\end{equation*}
where $\Delta(t) \coloneqq \|\overline{\mathsf{X}}^{(1)} - \overline{\mathsf{X}}^{(2)}\|_{t, L^1} + \|\overline{\mathsf{W}}^{(1)} - \overline{\mathsf{W}}^{(2)}\|_{t, \square}$ (assuming they share the same latent space).
Remarkably, this constant $C(t)$, inherited from the Lipschitz estimate of Lemma~\ref{lem:picard_estimates}, is independent of the initial state moment $\|\overline{\mathsf{X}}_0\|_{L^2}$.
Grönwall's inequality immediately yields the labeled stability: $\Delta(t) \leq e^{C(t)t} \Delta(0)$.

For the unlabeled stability, we take the infimum over all measure-preserving maps $\varphi_1, \varphi_2$ on both sides. 
Since the dynamics commute with relabeling (i.e., $\Psi(\overline{\mathcal{Y}}_\varphi) = \Psi(\overline{\mathcal{Y}})_\varphi$), the optimal coupling at $t=0$ propagates to time $t$, preserving the inequality for the unlabeled distance $\delta_{W_1, \square}$.
\end{proof}

\subsection{Proof of propagation of dissociatedness (Theorem~\ref{thm:propagation_of_independent_mixture})} \label{subsec:proof_propagation}

The proof strategy of Theorem~\ref{thm:propagation_of_independent_mixture} is closely related to the classical propagation of chaos arguments, such as those in \cite{sznitman1991topics}.
The core idea typically involves decomposing the error into a term controlled by Lipschitz continuity (often handled via Grönwall's inequality) and a fluctuation term of order $N^{-1/2}$ arising from independence.
For our setting, once an appropriate filtered probability space and coupling are defined, the subsequent calculations largely mirror the classical approach.

\begin{proof}[Proof of Theorem~\ref{thm:propagation_of_independent_mixture}]
The proof relies on a pathwise coupling between the finite system \eqref{eqn:main_SDEs} and the auxiliary processes, as outlined in Remark~\ref{rmk:finite_sampling}.
We consider the product probability space that explicitly supports both the random initial configurations and the driving Brownian motions for all $N$ agents:
\begin{equation*}
\begin{aligned}
\Big( (\Omega \times \widehat{\Omega})^N, (\mathcal{A} \otimes \widehat{\mathcal{F}})^{\otimes N}, \bigotimes_{i=1}^N (\pi_i \otimes \widehat{\pi} ) \Big).
\end{aligned}
\end{equation*}
Recall that the measures $(\pi_i)_{i\in[N]}$ satisfy the averaging condition $\frac{1}{N}\sum_{i=1}^N \pi_i = \pi$.
Consequently, the measure $\pi \otimes \widehat{\pi}$ on the limit space decomposes as $\frac{1}{N} \sum_{j=1}^N (\pi_j \otimes \widehat{\pi})$.
Using this, the interaction term in the McKean-Vlasov equation \eqref{eqn:main_McKean_Vlasov_X} can be rewritten as an average of expectations over the individual agent laws.
Specifically, for any $(\xi, \widehat{\xi}) \in \Omega \times \widehat{\Omega}$:
\begin{equation*}
\begin{aligned}
\mathrm{d}\overline{\mathsf{X}}_{\xi, \widehat{\xi}} &= \mu(\overline{\mathsf{X}}_{\xi, \widehat{\xi}}) \mathrm{d}t + \nu \mathrm{d}\overline{\mathsf{B}}_{\widehat{\xi}} + \frac{1}{N} \sum_{j \in [N]} \mathbb{E}_{(\zeta_j, \widehat{\zeta}_j) \sim \pi_j \otimes \widehat{\pi}} \Big[ \overline{\mathsf{W}}_{(\xi, \widehat{\xi}), (\zeta_j, \widehat{\zeta}_j)} \sigma(\overline{\mathsf{X}}_{\xi, \widehat{\xi}}, \overline{\mathsf{X}}_{\zeta_j, \widehat{\zeta}_j}) \Big] \mathrm{d}t.
\end{aligned}
\end{equation*}
Recall that the auxiliary processes $(\overline{\bm{X}}, \overline{\bm{W}})$ are defined via sampling in \eqref{eqn:sampling_definition}, restated here for convenience:
\begin{equation*}
\begin{aligned}
\overline{\bm{X}}_i(t) \coloneqq \overline{\mathsf{X}}_{\bm{\xi}_i, \widehat{\bm{\xi}}_i}(t), \qquad \overline{\bm{W}}_{i,j}(t) \coloneqq \overline{\mathsf{W}}_{(\bm{\xi}_i, \widehat{\bm{\xi}}_i), (\bm{\xi}_j, \widehat{\bm{\xi}}_j)}(t).
\end{aligned}
\end{equation*}
In this setup, for each agent $i \in [N]$, the measure $\pi_i$ on $\Omega$ governs the sampling of the initial structural label $\bm{\xi}_i$, while $\widehat{\pi}$ on $\widehat{\Omega}$ governs the canonical Brownian path $\widehat{\bm{\xi}}_i$ (associated with $\bm{B}_i$).

In the subsequent analysis, let $\bm{z}_i \coloneqq (\bm{\xi}_i, \widehat{\bm{\xi}}_i)$ denote the random variable for the $i$-th agent.
We also introduce independent copies $\bm{z}'_j \coloneqq (\bm{\zeta}_j, \widehat{\bm{\zeta}}_j) \sim \pi_j \otimes \widehat{\pi}$.
Crucially, $\bm{z}'_j$ is independent of the realization $\bm{z}_i$ coupled to the finite system, serving as the integration variable to ensure the expectation is taken over the law.

Subtracting the auxiliary dynamics from the finite system dynamics, we start with the weight component $\bm{W}$. 
Noting that $\overline{\bm{W}}_{i,j}(t) = \overline{\mathsf{W}}_{\bm{z}_i, \bm{z}_j}(t)$ and that the adaptation is local, we have for any $i \neq j$:
\begin{equation*}
\begin{aligned}
& \bm{W}_{i,j}(t) - \overline{\bm{W}}_{i,j}(t)
\\
& \quad = \int_0^t \Big( \alpha(\bm{X}_i(s), \bm{X}_j(s)) \bm{W}_{i,j}(s) - \alpha(\overline{\bm{X}}_i(s), \overline{\bm{X}}_j(s)) \overline{\bm{W}}_{i,j}(s) \Big) \mathrm{d}s
\\
& \qquad + \int_0^t \Big( \beta(\bm{X}_i(s), \bm{X}_j(s)) - \beta(\overline{\bm{X}}_i(s), \overline{\bm{X}}_j(s)) \Big) \mathrm{d}s.
\end{aligned}
\end{equation*}
It is straightforward to verify (by adding and subtracting cross terms and applying the triangle inequality) that
\begin{equation*}
\begin{aligned}
& | \bm{W}_{i,j}(t) - \overline{\bm{W}}_{i,j}(t)|
\\
& \quad \leq \Big( \|\nabla \beta\|_{L^\infty} + \|\nabla \alpha\|_{L^\infty} \|\overline{\mathsf{W}}\|_{L^\infty} \Big) \int_0^t \Big( | \bm{X}_i(s) - \overline{\bm{X}}_i(s) | + | \bm{X}_j(s) - \overline{\bm{X}}_j(s) | \Big) \mathrm{d}s
\\
& \qquad + \|\alpha\|_{L^\infty} \int_0^t | \bm{W}_{i,j}(s) - \overline{\bm{W}}_{i,j}(s)| \mathrm{d}s,
\end{aligned}
\end{equation*}
where we use the shorthand notation $\|\overline{\mathsf{W}}\|_{L^\infty}$ for the norm $\|\overline{\mathsf{W}}\|_{L^\infty([0,T] \times (\Omega \times \widehat{\Omega})^2)}$, whose a priori bound was established in Proposition~\ref{prop:well_posedness_McKean_Vlasov}.

Next, we turn to the state component $\bm{X}$.
Subtracting the auxiliary dynamics from the finite system dynamics, the Brownian motion terms cancel out.
We decompose the error into four distinct terms by adding and subtracting appropriate intermediate quantities:
\begin{equation*}
\begin{aligned}
\bm{X}_i(t) - \overline{\bm{X}}_i(t) = \int_0^t \Big( \mathcal{T}_{i,1}(s) + \mathcal{T}_{i,2}(s) + \mathcal{T}_{i,3}(s) + \mathcal{T}_{i,4}(s) \Big) \, \mathrm{d}s.
\end{aligned}
\end{equation*}
The first two terms capture the error propagation due to the Lipschitz continuity of the coefficients:
\begin{align*}
\mathcal{T}_{i,1}(s) &\coloneqq \mu(\bm{X}_i(s)) - \mu(\overline{\bm{X}}_i(s)),
\\
\mathcal{T}_{i,2}(s) &\coloneqq \frac{1}{N} \sum_{j \in [N] \setminus \{i\}} \Big( \bm{W}_{i,j}(s) \sigma(\bm{X}_i(s), \bm{X}_j(s)) - \overline{\bm{W}}_{i,j}(s) \sigma(\overline{\bm{X}}_i(s), \overline{\bm{X}}_j(s)) \Big).
\end{align*}
The third term represents the statistical fluctuation arising from the finite sampling of the auxiliary environment.
Crucially, for a fixed agent $i$, the summands indexed by $j$ are independent and centered:
\begin{equation*}
\mathcal{T}_{i,3}(s) \coloneqq \frac{1}{N} \sum_{j \in [N] \setminus \{i\}} \bigg( \overline{\bm{W}}_{i,j}(s) \sigma(\overline{\bm{X}}_i(s), \overline{\bm{X}}_j(s)) - \mathbb{E}_{\bm{z}'_j} \Big[ \overline{\mathsf{W}}_{\bm{z}_i, \bm{z}'_j}(s) \sigma(\overline{\mathsf{X}}_{\bm{z}_i}(s), \overline{\mathsf{X}}_{\bm{z}'_j}(s)) \Big] \bigg).
\end{equation*}
The final term accounts for the diagonal bias, resulting from the difference between the sum over $j \neq i$ and the full expectation average:
\begin{equation*}
\mathcal{T}_{i,4}(s) \coloneqq - \frac{1}{N} \mathbb{E}_{\bm{z}'_i} \Big[ \overline{\mathsf{W}}_{\bm{z}_i, \bm{z}'_i}(s) \sigma(\overline{\mathsf{X}}_{\bm{z}_i}(s), \overline{\mathsf{X}}_{\bm{z}'_i}(s)) \Big].
\end{equation*}

We first estimate the Lipschitz propagation terms $\mathcal{T}_{i,1}$ and $\mathcal{T}_{i,2}$.
By Assumption~\ref{assu:coefficients} (Lipschitz coefficients) and the uniform boundedness of the trajectories on $[0,T]$ (which ensures $\|\bm{W}\|_{L^\infty}$ is finite), we obtain the following pointwise bounds:
\begin{equation*}
\begin{aligned}
|\mathcal{T}_{i,1}(s)| &\leq \|\nabla \mu\|_{L^\infty} |\bm{X}_i(s) - \overline{\bm{X}}_i(s)|,
\\
|\mathcal{T}_{i,2}(s)| &\leq \frac{1}{N} \sum_{j \in [N] \setminus \{i\}} \Big( \|\sigma\|_{L^\infty} |\bm{W}_{i,j}(s) - \overline{\bm{W}}_{i,j}(s)| 
\\
& \qquad \qquad \qquad + \|\bm{W}\|_{L^\infty} \|\nabla \sigma\|_{L^\infty} \big( |\bm{X}_i(s) - \overline{\bm{X}}_i(s)| + |\bm{X}_j(s) - \overline{\bm{X}}_j(s)| \big) \Big)
\end{aligned}
\end{equation*}

For the fluctuation term $\mathcal{T}_{i,3}$, we employ the standard conditional independence argument found in the propagation of chaos literature (e.g., \cite{sznitman1991topics}).
Let us denote the summand by:
\begin{equation*}
\bm{V}_{i,j}(s) \coloneqq \overline{\mathsf{W}}_{\bm{z}_i, \bm{z}_j}(s) \sigma(\overline{\mathsf{X}}_{\bm{z}_i}(s), \overline{\mathsf{X}}_{\bm{z}_j}(s)) - \mathbb{E}_{\bm{z}'_j} \Big[ \overline{\mathsf{W}}_{\bm{z}_i, \bm{z}'_j}(s) \sigma(\overline{\mathsf{X}}_{\bm{z}_i}(s), \overline{\mathsf{X}}_{\bm{z}'_j}(s)) \Big].
\end{equation*}
Fix $i \in [N]$. Conditioned on the realization $\bm{z}_i$, the term $\bm{V}_{i,j}(s)$ depends solely on $\bm{z}_j$.
Since the auxiliary variables $(\bm{z}_j)_{j \neq i}$ are mutually independent and distributed according to $\pi_j \otimes \widehat{\pi}$ (matching the law of the integration variable $\bm{z}'_j$), the family $\{\bm{V}_{i,j}(s)\}_{j \in [N] \setminus \{i\}}$ consists of independent random variables with zero mean.
Consequently, when computing the conditional second moment of the sum, all cross-terms vanish:
\begin{equation*}
\begin{aligned}
\mathbb{E} \bigg[ \bigg( \sum_{j \in [N] \setminus \{i\}} \bm{V}_{i,j}(s) \bigg)^2 \, \Big| \, \bm{z}_i \bigg] 
= \sum_{j \in [N] \setminus \{i\}} \mathbb{E} \Big[ |\bm{V}_{i,j}(s)|^2 \, \Big| \, \bm{z}_i \Big]
\leq N \cdot \Big( 2 \|\overline{\mathsf{W}}\|_{L^\infty} \|\sigma\|_{L^\infty} \Big)^2.
\end{aligned}
\end{equation*}
Using Hölder's inequality to bound the expected absolute value by the square root of the second moment, we obtain the estimate for the integral:
\begin{equation*}
\begin{aligned}
\mathbb{E} \bigg[ \int_0^t |\mathcal{T}_{i,3}(s)| \, \mathrm{d}s \bigg]
&= \int_0^t \mathbb{E} \bigg[ \mathbb{E} \Big[ |\mathcal{T}_{i,3}(s)| \, \Big| \, \bm{z}_i \Big] \bigg] \, \mathrm{d}s
\\
&\leq \int_0^t \frac{1}{N} \mathbb{E} \bigg[ \bigg( \mathbb{E} \bigg[ \bigg( \sum_{j \in [N] \setminus \{i\}} \bm{V}_{i,j}(s) \bigg)^2 \, \Big| \, \bm{z}_i \bigg] \bigg)^{\frac{1}{2}} \bigg] \, \mathrm{d}s
\\
&\leq \int_0^t \frac{1}{N} \Big( N \cdot 4 \|\overline{\mathsf{W}}\|_{L^\infty}^2 \|\sigma\|_{L^\infty}^2 \Big)^{\frac{1}{2}} \, \mathrm{d}s
\\
&= 2 \|\overline{\mathsf{W}}\|_{L^\infty} \|\sigma\|_{L^\infty} \frac{t}{\sqrt{N}}.
\end{aligned}
\end{equation*}
Finally, for the diagonal term $\mathcal{T}_{i,4}$, the bound is straightforward due to the boundedness of the coefficients and the scaling factor:
\begin{equation*}
\int_0^t |\mathcal{T}_{i,4}(s)| \, \mathrm{d}s \leq \|\overline{\mathsf{W}}\|_{L^\infty} \|\sigma\|_{L^\infty} \frac{t}{N}.
\end{equation*}

Combining all the estimates above, we establish a Grönwall-type inequality for the maximum error.
Define the global error function:
\begin{equation*}
\begin{aligned}
G(t) \coloneqq \max\bigg\{ \sup_{i \in [N]} \mathbb{E} \Big[ |\bm{X}_i(t) - \overline{\bm{X}}_i(t)| \Big] , \quad \sup_{i \neq j} \mathbb{E} \Big[ |\bm{W}_{i,j}(t) - \overline{\bm{W}}_{i,j}(t)| \Big] \bigg\}.
\end{aligned}
\end{equation*}
From the estimates for $\mathcal{T}_{i,1}$ through $\mathcal{T}_{i,4}$ and the weight error dynamics, we have:
\begin{equation*}
\begin{aligned}
G(t) \leq \int_0^t \bigg( C_1 G(s) + \frac{C_2}{\sqrt{N}} \bigg) \, \mathrm{d}s,
\end{aligned}
\end{equation*}
where the constants can be chosen explicitly as:
\begin{equation*}
\begin{aligned}
C_1 &= (2 \|\nabla \alpha\|_{L^\infty} + 2\|\nabla \sigma\|_{L^\infty}) \|\overline{\mathsf{W}}\|_{L^\infty} + \|\alpha\|_{L^\infty} + 2 \|\nabla \beta\|_{L^\infty} 
+ \|\nabla \mu\|_{L^\infty} + \|\sigma\|_{L^\infty},
\\
C_2 &= 3\|\overline{\mathsf{W}}\|_{L^\infty} \|\sigma\|_{L^\infty}.
\end{aligned}
\end{equation*}
(Note that for $C_2$, we used the bound $N^{-1} \leq N^{-1/2}$ to combine the diagonal term with the fluctuation term).
Consequently, by Grönwall's inequality, we obtain the convergence rate:
\begin{equation*}
\begin{aligned}
G(t) \leq \frac{\exp(C_1 t) C_2 t}{\sqrt{N}}.
\end{aligned}
\end{equation*}
This completes the proof of Theorem~\ref{thm:propagation_of_independent_mixture}.

\end{proof}

\subsection{Proof of empirical stability (Theorem~\ref{thm:empirical_data_stability})} \label{subsec:proof_empirical}

In this subsection, we implement the proof of Theorem~\ref{thm:empirical_data_stability} by combining all the estimates derived previously, together with the Sampling Lemma (Lemma~\ref{lem:sampling_lemma}) whose proof is deferred to Appendix~\ref{subsec:proof_sampling}.

\begin{proof}[Proof of Theorem~\ref{thm:empirical_data_stability}]
Let $\bm{\mathcal{Y}}_N(t)$ be the solution to the finite system \eqref{eqn:main_SDEs} and $\overline{\mathcal{Y}}(t)$ be the target limit solution to \eqref{eqn:main_McKean_Vlasov}. 
We first proceed under the additional assumption that the initial configuration of $\bm{\mathcal{Y}}_N$ satisfies the dissociatedness property, bridging the gap relies on the two aforementioned objects:
\begin{enumerate}
\item The \textit{limit flow} $\overline{\mathcal{Y}}_N(t) = (\overline{\mathsf{X}}_N(t),\overline{\mathsf{W}}_N(t))$, solving \eqref{eqn:main_McKean_Vlasov} initialized from the Aldous-Hoover representation of $\bm{\mathcal{Y}}_N(0) = ((\bm{X}_{i}(0))_{i \in [N]}, (\bm{W}_{i,j}(0))_{i,j \in [N]})$.
\item The \textit{auxiliary particle system} $\overline{\bm{\mathcal{Y}}}_N(t)$, defined as the sampling of $\overline{\mathcal{Y}}_N(t)$ (coupled with $\bm{\mathcal{Y}}_N(t)$ via the same initial data and Brownian motions).
\end{enumerate}
By the triangle inequality, the total error admits the following decomposition, where each term corresponds to an explicit bound in our discussion:
\begin{equation}\label{eqn:error_decomposition}
\begin{aligned}
& \mathbb{E} \Big[ \delta_{W_1, \square}\big( \bm{\mathcal{Y}}_N(t), \overline{\mathcal{Y}}(t) \big) \Big] 
\\
\quad &\leq \underbrace{\mathbb{E} \Big[ \delta_{W_1, \square}\big( \bm{\mathcal{Y}}_N(t), \overline{\bm{\mathcal{Y}}}_N(t) \big) \Big]}_{\substack{\text{Thm.~\ref{thm:propagation_of_independent_mixture} - Sec.~\ref{subsec:proof_propagation}} \\ \text{Propagation of Dissociatedness}}} + \underbrace{\mathbb{E} \Big[ \delta_{W_1, \square}\big( \overline{\bm{\mathcal{Y}}}_N(t), \overline{\mathcal{Y}}_N(t) \big) \Big]}_{\substack{\text{Lem.~\ref{lem:sampling_lemma} - Appendix~\ref{subsec:proof_sampling}} \\ \text{Sampling Lemma}}}
+ \underbrace{\vphantom{\Big[} \delta_{W_1, \square}\big( \overline{\mathcal{Y}}_N(t), \overline{\mathcal{Y}}(t) \big) }_{\substack{\text{Prop.~\ref{prop:well_posedness_McKean_Vlasov} - Sec.~\ref{subsec:proof_well_posedness}} \\ \text{McKean-Vlasov Stability}}}
\\
\quad &\leq \frac{C_1(t)}{\sqrt{N}} + \frac{C_2(t)}{\sqrt{\log N}} + C_3(t) \delta_{W_1, \square}\big( \overline{\mathcal{Y}}_N(0), \overline{\mathcal{Y}}(0) \big).
\end{aligned}
\end{equation}
This estimate is stronger than \eqref{eqn:empirical_data_stability} in Theorem~\ref{thm:empirical_data_stability}.

To establish \eqref{eqn:empirical_data_stability} for general initial data (where the initial configuration of $\bm{\mathcal{Y}}_N$ may not be dissociated), we proceed by conditioning on the initial realization $\bm{\mathcal{Y}}_N(0)$. 
Let $\mathcal{F}_0$ denote the $\sigma$-algebra generated by the initial configuration.
Conditional on $\mathcal{F}_0$, the initial data is deterministic. 
Recall that any deterministic configuration on a set of $N$ agents is trivially dissociated and can be viewed as an Aldous-Hoover representation on the probability space $\Omega = [N]$ equipped with the uniform measure. 
Consequently, the error decomposition derived above holds pathwise for the conditional expectation:
\begin{equation*}
\begin{aligned}
\mathbb{E} \Big[ \delta_{W_1, \square}\big( \bm{\mathcal{Y}}_N(t), \overline{\mathcal{Y}}(t) \big) \,\Big|\, \mathcal{F}_0 \Big] 
\leq \frac{\bm{C}_1(t)}{\sqrt{N}} + \frac{\bm{C}_2(t)}{\sqrt{\log N}} + \bm{C}_3(t) \delta_{W_1, \square}\big( \bm{\mathcal{Y}}_N(0), \overline{\mathcal{Y}}(0) \big).
\end{aligned}
\end{equation*}

To derive the final estimate \eqref{eqn:empirical_data_stability}, we take the full expectation on both sides of \eqref{eqn:error_decomposition}.
However, we must verify that the coefficients $\bm{C}_1(t)$, $\bm{C}_2(t)$, and $\bm{C}_3(t)$, which now depend on the specific realization of the initial moments, remain bounded after taking expectation.

First, regarding $\bm{C}_1(t)$ (from the propagation of dissociatedness in Section~\ref{subsec:proof_propagation}) and $\bm{C}_3(t)$ (from the McKean-Vlasov stability in Section~\ref{subsec:proof_well_posedness}), we observe that they depend on the initial data primarily through the $L^\infty$-bound of the weights, but are independent of the $L^2$-moment of the states.
By Assumption~\ref{assu:initial_data}, the initial weights are uniformly bounded almost surely.
Consequently, $\bm{C}_1(t)$ and $\bm{C}_3(t)$ are uniform with respect to the randomness of the initial configuration.

Second, regarding the sampling error term, Lemma~\ref{lem:sampling_lemma} indicates that the coefficient $\bm{C}_2(t)$ takes the form
\begin{equation*}
C \left( \|\overline{\mathsf{X}}_N(t)\|_{L^2} + \|\overline{\mathsf{W}}_N(t)\|_{L^\infty} \right) = C \ \mathbb{E} \left[ \left( \frac{1}{N} \sum_{i=1}^N |\overline{\bm{X}}_i(t)|^2 \right)^{1/2} + \max_{i,j} |\overline{\bm{W}}_{i,j}(t)| \,\bigg|\, \mathcal{F}_0 \right].
\end{equation*}
Under Assumption~\ref{assu:initial_data} and the moments propagation in Section~\ref{subsec:proof_well_posedness}, the second moment of the states and the infinity norm of the weights are bounded in full expectation.
Therefore, taking the full expectation of \eqref{eqn:error_decomposition} yields the bound \eqref{eqn:empirical_data_stability}, completing the proof.

\end{proof}

\appendix



\section{Graphon Theory Tools (Proof of Sampling Lemma~\ref{lem:sampling_lemma})} \label{subsec:proof_sampling}

The proof presented here is a natural extension of the argument for the Sampling Lemma found in the textbook \cite{lovasz2012large}. We include it for the completeness of this article.

The primary modifications involve our consideration of sampling via independent latent variables which are not necessarily identically distributed. Additionally, we need to address the presence of the state function $\mathsf{X}: \Omega \to \mathbb{D}$ alongside the weight kernel $\mathsf{W}: \Omega \times \Omega \to \mathbb{R}$. However, these adjustments do not alter the fundamental logic of the combinatorial arguments or the large deviation estimates.

Recall the setting of Lemma~\ref{lem:sampling_lemma}. We consider a sequence of $N$ independent latent variables $\bm{\xi} = (\bm{\xi}_i)_{i \in [N]}$ taking values in $\Omega$. Their distributions $\pi_i = \mathrm{Law}(\bm{\xi}_i)$ are assumed to be absolutely continuous with respect to $\pi$ and satisfy the averaging condition $\frac{1}{N}\sum_{i=1}^N \pi_i = \pi$.
The sampled empirical structure $\bm{\mathcal{Y}}_N = (\bm{X}, \bm{W})$ is generated by:
\begin{equation*}
    \bm{X}_i = \mathsf{X}(\bm{\xi}_i), \quad \bm{W}_{i,j} = \mathsf{W}(\bm{\xi}_i, \bm{\xi}_j).
\end{equation*}
For notational conciseness, we denote the collection of these sampled random variables by:
\begin{equation*}
    \mathsf{X}_{[\bm{\xi}]} \coloneqq (\mathsf{X}(\bm{\xi}_i))_{i \in [N]}, \quad \mathsf{W}_{[\bm{\xi}]} \coloneqq (\mathsf{W}(\bm{\xi}_i, \bm{\xi}_j))_{i,j \in [N]}.
\end{equation*}
The proof strategy relies on introducing an intermediate piecewise constant approximation. Let $\mathcal{P} = \{S_1, \dots, S_m\}$ be a partition of the latent space $\Omega$ into $m$ measurable sets with positive measures. We define the piecewise constant projections $\mathsf{X}_{\mathcal{P}}: \Omega \to \mathbb{D}$ and $\mathsf{W}_{\mathcal{P}}: \Omega \times \Omega \to \mathbb{R}$ by averaging the original functions over the partition blocks:
\begin{align*}
    \mathsf{X}_{\mathcal{P}}(\xi) &\coloneqq \sum_{k=1}^m \mathbbm{1}_{S_k}(\xi) \left( \frac{1}{\pi(S_k)} \int_{S_k} \mathsf{X}(\zeta) \, \mathrm{d}\pi(\zeta) \right), \\
    \mathsf{W}_{\mathcal{P}}(\xi, \zeta) &\coloneqq \sum_{k,l=1}^m \mathbbm{1}_{S_k}(\xi) \mathbbm{1}_{S_l}(\zeta) \left( \frac{1}{\pi(S_k)\pi(S_l)} \int_{S_k \times S_l} \mathsf{W}(\xi', \zeta') \, \mathrm{d}\pi(\xi') \mathrm{d}\pi(\zeta') \right).
\end{align*}
Correspondingly, we define the sampled versions of these projections, denoted by $\mathsf{X}_{\mathcal{P}[\bm{\xi}]}$ and $\mathsf{W}_{\mathcal{P}[\bm{\xi}]}$, by evaluating the \textit{already projected} functions at the sample points:
\begin{equation*}
    \mathsf{X}_{\mathcal{P}[\bm{\xi}]} \coloneqq (\mathsf{X}_{\mathcal{P}}(\bm{\xi}_i))_{i \in [N]}, \quad \mathsf{W}_{\mathcal{P}[\bm{\xi}]} \coloneqq (\mathsf{W}_{\mathcal{P}}(\bm{\xi}_i, \bm{\xi}_j))_{i,j \in [N]}.
\end{equation*}
This construction allows us to decompose the total sampling error via the triangle inequality:
\begin{equation} \label{eqn:sampling_decomposition}
\begin{aligned}
    \delta_{W_1, \square} \big( (\mathsf{X}_{[\bm{\xi}]}, \mathsf{W}_{[\bm{\xi}]}), (\mathsf{X}, \mathsf{W}) \big)
    &\leq \delta_{W_1, \square} \big( (\mathsf{X}_{[\bm{\xi}]}, \mathsf{W}_{[\bm{\xi}]}), (\mathsf{X}_{\mathcal{P}[\bm{\xi}]}, \mathsf{W}_{\mathcal{P}[\bm{\xi}]}) \big) \\
    &\quad + \delta_{W_1, \square} \big( (\mathsf{X}_{\mathcal{P}[\bm{\xi}]}, \mathsf{W}_{\mathcal{P}[\bm{\xi}]}), (\mathsf{X}_{\mathcal{P}}, \mathsf{W}_{\mathcal{P}}) \big) \\
    &\quad + \delta_{W_1, \square} \big( (\mathsf{X}_{\mathcal{P}}, \mathsf{W}_{\mathcal{P}}), (\mathsf{X}, \mathsf{W}) \big).
\end{aligned}
\end{equation}
The proof then proceeds by bounding each term: the last term represents the deterministic approximation error $(\mathsf{X}-\mathsf{X}_{\mathcal{P}}, \mathsf{W}-\mathsf{W}_{\mathcal{P}})$, which can be computed using the labeled distance; the first term is analogous in nature but involves the concentration of sampling for the residuals; and the middle term reduces to a combinatorial problem of comparing partition volumes $\pi(S_k)$ with their empirical counts.

The logic for the concentration of sampling follows the arguments in \cite{lovasz2012large} based on Azuma-Hoeffding bounds. However, our current focus is not on replicating the precise convergence rates found in the literature. Instead, to describe estimates that hold with overwhelming probability, we introduce the following notion:

\begin{defi}
If, for a sequence of random events $A_k$, $k \to \infty$, there exist constants $C_1 > 0, C_2 > 0,$ and $\theta > 0$ such that for all $k$ (or for all $k \geq k_0$ for some initial $k_0 \in \N$),
\begin{equation*}
\Pb(A_k) \geq 1 - C_1 \exp(-C_2 k^{\theta}),
\end{equation*}
we say that the sequence of events $A_k$, $k \to \infty$ is $k$-certain.
\end{defi}

With the definition of $k$-certainty, we can now state our adapted version of the first sampling lemma.

\begin{lem}[First Sampling Lemma adapted from \texorpdfstring{\cite[Lemma~10.6]{lovasz2012large}}{[RefName, Lemma 10.6]}]
\label{lem:first_sampling_kernels}
Let $\mathsf{W} \in L^\infty(\Omega^2)$ and $\mathsf{X} \in L^2(\Omega)$. Let $\bm{\xi} = (\bm{\xi}_i)_{i \in [N]} \in \Omega^N$ be a sequence of independent latent variables with distributions $\pi_i$ satisfying $\frac{1}{N}\sum_{i=1}^N \pi_i = \pi$. Then, it is $N$-certain that
\begin{equation*}
\Big| \|\mathsf{W}_{[\bm{\xi}]}\|_{\square} - \|\mathsf{W}\|_{\square} \Big| \leq \frac{C \|\mathsf{W}\|_{L^\infty}}{N^{1/4}}
\quad \text{and} \quad
\Big| \|\mathsf{X}_{[\bm{\xi}]}\|_{L^1} - \|\mathsf{X}\|_{L^1} \Big| \leq \frac{C \|\mathsf{X}\|_{L^2}}{N^{1/4}},
\end{equation*}
where $C > 0$ is a universal constant.
\end{lem}

\begin{proof}[Sketch of proof]
The proof for the cut norm concentration follows the argument of \cite[Lemma 10.6]{lovasz2012large} almost verbatim. Although the original lemma is stated for symmetric kernels and i.i.d. sampling, a step-by-step verification reveals that neither symmetry nor the identical distribution property is utilized in the derivation; the core concentration inequalities (e.g., Azuma-Hoeffding) rely solely on the independence of the samples.
For the state function $\mathsf{X}$, the quantity $\|\mathsf{X}_{[\bm{\xi}]}\|_{L^1}$ is simply the empirical average of independent random variables $|\mathsf{X}(\bm{\xi}_i)|$. While standard concentration results (e.g., related to the Central Limit Theorem) readily yield a tighter rate of order $O(N^{-1/2})$, we adopt the unified $N^{-1/4}$ bound here to match the cut norm estimate, which suffices for our definition of $N$-certainty.
\end{proof}

We now prove Lemma~\ref{lem:sampling_lemma} using Lemma~\ref{lem:first_sampling_kernels}.

\begin{proof}[Proof of Lemma~\ref{lem:sampling_lemma}]
Let $m = \lceil N^{1/4} \rceil$. One can construct (following the procedure in \cite[Lemma~9.15]{lovasz2012large}) an equipartition $\mathcal{P} = \{S_1,\dots, S_m\}$ of $\Omega$ into $m$ classes such that
\begin{equation*}
\begin{aligned}
\|\mathsf{W} - \mathsf{W}_{\mathcal{P}}\|_{\square} \leq \frac{8 \|\mathsf{W}\|_{L^\infty}}{\sqrt{\log N}}, \quad \|\mathsf{X} - \mathsf{X}_{\mathcal{P}}\|_{L^1} \leq \frac{8 \|\mathsf{X}\|_{L^2}}{\sqrt{\log N}}.
\end{aligned}
\end{equation*}
Applying the decomposition \eqref{eqn:sampling_decomposition} with this specific partition $\mathcal{P}$, we proceed to bound the three resulting terms.
For the last term, since the distance $\delta_{W_1,\square}$ is bounded by the labeled distance on the common space $\Omega$, the construction of $\mathcal{P}$ directly implies:
\begin{equation*}
\begin{aligned}
\delta_{W_1,\square} \left( (\mathsf{X}_{\mathcal{P}}, \mathsf{W}_{\mathcal{P}}), (\mathsf{X}, \mathsf{W}) \right) 
\leq \|\mathsf{X} - \mathsf{X}_{\mathcal{P}}\|_{L^1} + \|\mathsf{W} - \mathsf{W}_{\mathcal{P}}\|_{\square} 
\leq \frac{C (\|\mathsf{X}\|_{L^2} + \|\mathsf{W}\|_{L^\infty})}{\sqrt{\log N}}.
\end{aligned}
\end{equation*}

Next, we consider the first term in \eqref{eqn:sampling_decomposition}. Note that the distance between the sampled structures is bounded by the norms of their differences:
\begin{equation*}
\begin{aligned}
\delta_{W_1, \square} \big( (\mathsf{X}_{[\bm{\xi}]}, \mathsf{W}_{[\bm{\xi}]}), (\mathsf{X}_{\mathcal{P}[\bm{\xi}]}, \mathsf{W}_{\mathcal{P}[\bm{\xi}]}) \big)
&\leq \|\mathsf{X}_{[\bm{\xi}]} - \mathsf{X}_{\mathcal{P}[\bm{\xi}]}\|_{L^1} + \|\mathsf{W}_{[\bm{\xi}]} - \mathsf{W}_{\mathcal{P}[\bm{\xi}]}\|_{\square} \\
&= \|(\mathsf{X} - \mathsf{X}_{\mathcal{P}})_{[\bm{\xi}]}\|_{L^1} + \|(\mathsf{W} - \mathsf{W}_{\mathcal{P}})_{[\bm{\xi}]}\|_{\square}.
\end{aligned}
\end{equation*}
We apply Lemma~\ref{lem:first_sampling_kernels} to the residual functions $\mathsf{X} - \mathsf{X}_{\mathcal{P}}$ and $\mathsf{W} - \mathsf{W}_{\mathcal{P}}$. It follows that it is $N$-certain that the sampled norms are close to the true norms:
\begin{equation*}
\begin{aligned}
\|(\mathsf{W} - \mathsf{W}_{\mathcal{P}})_{[\bm{\xi}]}\|_{\square}
&\leq \|\mathsf{W} - \mathsf{W}_{\mathcal{P}}\|_{\square} + \frac{C \|\mathsf{W} - \mathsf{W}_{\mathcal{P}}\|_{L^\infty}}{N^{1/4}}
\\
&\leq \frac{C \|\mathsf{W}\|_{L^\infty}}{\sqrt{\log N}} + \frac{2 C \|\mathsf{W}\|_{L^\infty}}{N^{1/4}}.
\end{aligned}
\end{equation*}
A similar estimate holds for the state component $\mathsf{X}$ (using the $L^1$ norm concentration). Since $N^{-1/4}$ decays significantly faster than $(\log N)^{-1/2}$, and the failure probability of the $N$-certain event decays exponentially (specifically, as $\exp(-C N^{\theta})$, which is negligible compared to any polynomial inverse), the contribution of the ``bad events'' to the expectation is asymptotically vanishing. Consequently, the total bound for the first term is dominated by the logarithmic rate:
\begin{equation*}
\begin{aligned}
\mathbb{E} \Big[ \delta_{W_1, \square} \big( (\mathsf{X}_{[\bm{\xi}]}, \mathsf{W}_{[\bm{\xi}]}), (\mathsf{X}_{\mathcal{P}[\bm{\xi}]}, \mathsf{W}_{\mathcal{P}[\bm{\xi}]}) \big) \Big]
\leq \frac{C (\|\mathsf{X}\|_{L^2} + \|\mathsf{W}\|_{L^\infty})}{\sqrt{\log N}}.
\end{aligned}
\end{equation*}

Finally, we handle the middle term in \eqref{eqn:sampling_decomposition}, which measures the distance between the projected structure and its sampled version. Since $\mathsf{X}_{\mathcal{P}}$ and $\mathsf{W}_{\mathcal{P}}$ are constant on the blocks of $\mathcal{P} = \{S_1, \dots, S_m\}$, this distance is controlled by the discrepancy between the volumes of the blocks $\pi(S_k)$ and their empirical frequencies.
Let $r_k \coloneqq \frac{1}{N} \sum_{i \in [N]} \mathbbm{1}_{S_k}(\bm{\xi}_i) - \pi(S_k)$.
To bound the distance $\delta_{W_1, \square}$, we construct a coupling that matches the empirical mass on each block $S_k$ with the measure $\pi|_{S_k}$ to the maximal extent possible. The remaining unmatched mass, which corresponds to the discrepancy $\sum_{k \in [m]} |r_k|$, is coupled arbitrarily. Estimating the cost for this residual portion (linearly for the $L^\infty$-bounded weights and via the Cauchy-Schwarz inequality for the $L^2$-bounded states) yields the following bound for the combinatorial error:
\begin{equation*}
\begin{aligned}
\delta_{W_1, \square} \big( (\mathsf{X}_{\mathcal{P}[\bm{\xi}]}, \mathsf{W}_{\mathcal{P}[\bm{\xi}]}), (\mathsf{X}_{\mathcal{P}}, \mathsf{W}_{\mathcal{P}}) \big)
\leq \|\mathsf{W}\|_{L^\infty} \sum_{k \in [m]} |r_k| + \|\mathsf{X}\|_{L^2} \sqrt{\sum_{k \in [m]} |r_k|}.
\end{aligned}
\end{equation*}
Taking the expectation, standard concentration estimates for the multinomial distribution imply that the mean deviation $\mathbb{E}[\sum |r_k|]$ scales as $O(m/\sqrt{N})$. With our choice $m \sim N^{1/4}$, this yields a polynomial decay rate asymptotically negligible compared to the logarithmic rate $(\log N)^{-1/2}$. 

Combining the estimates for all three terms yields the result in \eqref{eqn:sampling_lemma}.

\end{proof}

\bigskip 
\noindent \textbf{Proof of Lemma~\ref{lem:first_sampling_kernels}.} 
\textit{Note: This detailed proof is included in the arXiv preprint for the convenience of the reader and to ensure self-containedness. It essentially adapts the standard arguments from dense graph limit theory (specifically, the proof of the Sampling Lemma in \cite[Lemma 10.6]{lovasz2012large}) to our setting with asymmetric kernels and non-identical sampling distributions. This section is omitted from the journal submission.}

\medskip


We begin by considering the positive one-sided cut norm, defined as
\begin{equation*}
\|\mathsf{W}\|_{\square}^{+} \coloneqq \sup_{S,T \subseteq \Omega} \int_{S \times T} \mathsf{W}(\xi,\zeta) \, \mathrm{d} \xi \, \mathrm{d} \zeta.
\end{equation*}
Given the identity $\|\mathsf{W}\|_{\square} = \max\left\{\|\mathsf{W}\|_{\square}^{+}, \|-\mathsf{W}\|_{\square}^{+}\right\}$, the lemma follows if we can show that it is $N$-certain that
\begin{equation*}
\Big| \|\mathsf{W}_{[\bm{\xi}]}\|_{\square}^{+} - \|\mathsf{W}\|_{\square}^{+} \Big| \leq \frac{C \|\mathsf{W}\|_{L^\infty}}{N^{1/4}},
\end{equation*}
where the constant $C$ may differ from that in Lemma~\ref{lem:first_sampling_kernels} and vary from line to line.

To proceed with the proof for the one-sided norm, we first adapt a supporting lemma concerning the selection of a random subset of column indices. Let $B = (B_{ij})_{i,j \in [N]}$ be an $N \times N$ matrix. For any set $Q_1 \subseteq [N]$ of rows and any set $Q_2 \subseteq [N]$ of columns, we set $B(Q_1, Q_2) \coloneqq \sum_{i \in Q_1, j \in Q_2} B_{ij}$.
For $Q_1 \subseteq [N]$ a set of rows, let
\begin{equation*}
Q_1^+ \coloneqq \{j \in [N] \mid B(Q_1, \{j\}) > 0\},
\end{equation*}
which implies $B(Q_1, Q_1^+) \geq 0$. The sets $Q_1^-$, $Q_2^+$, and $Q_2^-$ are defined analogously.

\begin{lem}[Random Column Subset, analogous to \texorpdfstring{\cite[Lemma~10.7]{lovasz2012large}}{[RefName, Lemma 10.7]}]
\label{lem:random_column_subset}
Let $B = (B_{ij})_{i,j \in [N]}$ be an $N \times N$ matrix.
Let $S_1, S_2 \subseteq [N]$ (where $S_1$ is a set of rows and $S_2$ is a set of columns). Let $\bm{Q}$ be a random $q$-subset of $[N]$ (representing selected column indices), where $1 \leq q \leq N$. Then
\begin{equation*}
B(S_1, S_2) \leq \E \left[ B( (\bm{Q} \cap S_2)^+, S_2) \right] + \frac{\|B\|_{L^\infty}^{} N^2}{\sqrt{q}}.
\end{equation*}
\end{lem}
The essence of this lemma is to provide a uniform upper bound, so the left-hand side can be strengthened to $B(S_2^+, S_2) \geq B(S_1, S_2)$.

\begin{proof}
We start from the following identity: For any $\bm{Q} \subseteq [N]$,
\begin{align*}
B(S_1,S_2) &= B( (\bm{Q} \cap S_2)^+ \cap S_1, S_2) + B( (\bm{Q} \cap S_2)^- \cap S_1, S_2)
\\
&= B( (\bm{Q} \cap S_2)^+, S_2)
\\
& \quad - \sum_{i \in [N] \setminus S_1} B( (\bm{Q} \cap S_2)^+ \cap \{i\}, S_2) + \sum_{i \in S_1} B( (\bm{Q} \cap S_2)^- \cap \{i\}, S_2).
\end{align*}
Our goal is to show the per-row bounds that for each $i \in [N]$,
\begin{align*}
\E \left[ B( (\bm{Q} \cap S_2)^- \cap \{i\}, S_2) \right] &\leq \frac{\|B\|_{L^\infty}^{} N}{\sqrt{q}}, \\
- \E \left[ B( (\bm{Q} \cap S_2)^+ \cap \{i\}, S_2) \right] &\leq \frac{\|B\|_{L^\infty}^{} N}{\sqrt{q}}.
\end{align*}
Taking the expectation for the identity and summing over $i$ then yields the desired estimate.

For a fixed row $i \in S_1$, let $\bm{A}_{i} \coloneqq \sum_{j \in \bm{Q} \cap S_2} B_{ij}$, $b_i \coloneqq \sum_{j \in S_2} B_{ij}$, and $c_i \coloneqq \sum_{j \in S_2} B_{ij}^2$. The expectation term for row $i$ is then
\begin{align*}
\E \left[ B\big( (\bm{Q} \cap S_2)^- \cap \{i\}, S_2\big) \right] = \Pb(\bm{A}_i < 0) b_i.
\end{align*}
If $b_i \leq 0$, this term is non-positive, satisfying the required bound. Assume $b_i > 0$.
We have moment estimates $\E[\bm{A}_i] = \frac{q}{N} b_i > 0$ and
\begin{equation*}
\E[\bm{A}_i^2] = \frac{q}{N} c_i + \frac{q(q-1)}{N(N-1)} (b_i^2 - c_i) \leq \frac{q}{N} c_i + \frac{q^2}{N^2} b_i^2.
\end{equation*}
By Chebyshev's inequality, since $\E[\bm{A}_i]>0$:
\begin{align*}
\Pb(\bm{A}_i < 0) \leq \Pb(|\bm{A}_i - \E[\bm{A}_i]| \geq \E[\bm{A}_i]) 
\leq \frac{\E[\bm{A}_i^2] - (\E[\bm{A}_i])^2}{(\E[\bm{A}_i])^2} \leq \frac{N}{q}\frac{c_i}{b_i^2}.
\end{align*}
Since $\Pb(\bm{A}_i < 0) \in [0,1]$ and $b_i > 0$, we have $\Pb(\bm{A}_i < 0)b_i \leq \sqrt{\Pb(\bm{A}_i < 0)} b_i$. Thus,
\begin{align*}
\E \left[ B\big( (\bm{Q} \cap S_2)^- \cap \{i\}, S_2\big) \right] \leq \sqrt{\Pb(\bm{A}_i < 0)} b_i \leq \frac{\|B\|_{L^\infty}^{} N}{\sqrt{q}}.
\end{align*}
This establishes the first per-row bound, and a symmetric argument for the second per-row bound completes the proof.

\end{proof}

The following lemma gives an upper bound on the one-sided cut norm, using the sampling procedure from the previous lemma.
\begin{lem}[One-Sided Upper Bound, analogous to \texorpdfstring{\cite[Lemma~10.8]{lovasz2012large}}{[RefName, Lemma 10.8]}]
\label{lem:upper_bound_one_sided_cut_norm}
Let $B = (B_{ij})_{i,j \in [N]}$ be an $N \times N$ matrix. Let $\bm{Q}_{1}$ and $\bm{Q}_{2}$ be independent random $q$-subsets of $[N]$ ($1\le q\le N$). Then
\begin{equation*}
\|B\|_{\square}^{+} \leq \frac{1}{N^{2}}\E_{\bm{Q}_{1},\bm{Q}_{2}}\left[\max_{\substack{R_{1}\subseteq \bm{Q}_{1}, R_{2}\subseteq \bm{Q}_{2}}} B(R_{2}^{+},R_{1}^{+})\right] + \frac{2 \|B\|_{L^\infty}^{}}{\sqrt{q}}.
\end{equation*}
\end{lem}

\begin{proof}
Let $S_1^*, S_2^* \subseteq [N]$ be sets such that $N^2 \|B\|_{\square}^{+} = B(S_1^*, S_2^*)$.
By applying Lemma~\ref{lem:random_column_subset} twice (once for columns and then symmetrically for rows), we have
\begin{align*}
B(S_1^*, S_2^*) & \leq \E_{\bm{Q}_2}\left[B\left((\bm{Q}_2 \cap S_2^*)^{+}, S_2^*\right)\right] + \frac{\|B\|_{L^\infty}^{} N^2}{\sqrt{q}}
\\
& \leq \E_{\bm{Q}_1,\bm{Q}_2}\left[B\left((\bm{Q}_2 \cap S_2^*)^{+}, (\bm{Q}_1 \cap (\bm{Q}_2 \cap S_2^*)^{+})^{+}\right)\right] + \frac{2 \|B\|_{L^\infty}^{} N^2}{\sqrt{q}}.
\end{align*}
Dividing by $N^2$ and taking the supremum over all $R_1 \subseteq \bm{Q}_1$ and $R_2 \subseteq \bm{Q}_2$ yields the lemma.
\end{proof}

\begin{proof}[Proof of Lemma~\ref{lem:first_sampling_kernels}]
To bound the difference $\|\mathsf{W}_{[\bm{\xi}]}\|_{\square} - \|\mathsf{W}\|_{\square}$ from below, we first consider its expectation.
For any two measurable subsets $S_{1}, S_{2}\subset\Omega$, we have by definition of the empirical cut norm,
\begin{equation*}
\|\mathsf{W}_{[\bm{\xi}]}\|_{\square}^{+} \ge \frac{1}{N^{2}}\sum_{i,j : \bm{\xi}_i \in S_1, \bm{\xi}_j \in S_2} \mathsf{W}(\bm{\xi}_i, \bm{\xi}_j).
\end{equation*}
Taking the expectation, assuming $\bm{\xi}_i$ are independent with $\mathrm{Law}(\bm{\xi}_i) = \pi_i$, we have
\begin{align*}
\E_{[\bm{\xi}]}\left[\|\mathsf{W}_{[\bm{\xi}]}\|_{\square}^{+}\right] &\ge \frac{1}{N^{2}}\E_{[\bm{\xi}]}\left[\sum_{i,j : \bm{\xi}_i \in S_1, \bm{\xi}_j \in S_2} \mathsf{W}(\bm{\xi}_i, \bm{\xi}_j)\right]
\\
&= \frac{1}{N^2} \sum_{i,j \in [N], \ i \neq j} \int_{S_{1}\times S_{2}} \mathsf{W}(\zeta, \eta) \, \mathrm{d} \pi_i(\zeta) \, \mathrm{d} \pi_j(\eta) + \frac{1}{N^2} \sum_{i \in [N]} \int_{S_{1}\cap S_{2}} \mathsf{W}(\zeta, \zeta) \, \mathrm{d} \pi_i(\zeta)
\\
&\ge \int_{S_{1}\times S_{2}} \mathsf{W}(\zeta, \eta) \, \mathrm{d} \pi(\zeta) \, \mathrm{d} \pi(\eta) - \frac{2 \|\mathsf{W}\|_{L^\infty}}{N}.
\end{align*}
Taking the supremum of the right side over all measurable sets $S_{1}, S_{2}$ yields
\begin{equation*}
\E_{[\bm{\xi}]}\left[\|\mathsf{W}_{[\bm{\xi}]}\|_{\square}^{+}\right] \ge \|\mathsf{W}\|_{\square}^{+} - \frac{2\|\mathsf{W}\|_{L^\infty}}{N}.
\end{equation*}
The $N$-certain bound $\|\mathsf{W}_{[\bm{\xi}]}\|_{\square}^{+} - \|\mathsf{W}\|_{\square}^{+} \geq -\frac{C \|\mathsf{W}\|_{L^\infty}}{N^{1/4}}$ then follows by a concentration of measure argument (such as Azuma's Inequality), noting that changing one sample $\bm{\xi}_k$ alters $\|\mathsf{W}_{[\bm{\xi}]}\|_{\square}^{+}$ by at most $O(1/N)$, which allows for the construction of a Doob martingale.

To prove the upper bound of $\|\mathsf{W}_{[\bm{\xi}]}\|_{\square}^{+} - \|\mathsf{W}\|_{\square}^{+}$, let $\bm{Q}_{1}$ and $\bm{Q}_{2}$ be independent random $q$-subsets of $[N]$, where $q=\lfloor\sqrt{N}/4\rfloor$. Lemma~\ref{lem:upper_bound_one_sided_cut_norm} states that for any given realization $[\bm{\xi}]$,
\begin{equation*}
\|\mathsf{W}_{[\bm{\xi}]}\|_{\square}^{+} \leq \frac{1}{N^{2}}\E_{\bm{Q}_{1},\bm{Q}_{2}}\left[\max_{\substack{R_{1}\subseteq \bm{Q}_{1}, R_{2}\subseteq \bm{Q}_{2} }} \mathsf{W}_{[\bm{\xi}]}(R_{2}^{+},R_{1}^{+})\right] + \frac{2 \|\mathsf{W}\|_{L^\infty}^{}}{\sqrt{q}}.
\end{equation*}

Let $\bm{Q} = \bm{Q}_1 \cup \bm{Q}_2$, and let $[\bm{\xi}]_{\bm{Q}} = (\bm{\xi}_i)_{i \in \bm{Q}}$ denote the subset of samples indexed by $\bm{Q}$.
For fixed realizations of $\bm{Q}_1$ and $\bm{Q}_2$, we consider arbitrary subsets $R_1 \subseteq \bm{Q}_1$ and $R_2 \subseteq \bm{Q}_2$. Accordingly, we define the sets $Y_1, Y_2 \subseteq \Omega$ as follows:
\begin{align*}
Y_1 = \left\{\zeta \in \Omega \ \middle|\ \sum_{i' \in R_1} \mathsf{W}(\bm{\xi}_{i'},\zeta) > 0\right\}, \quad
Y_2 = \left\{\zeta \in \Omega \ \middle|\ \sum_{j' \in R_2} \mathsf{W}(\zeta, \bm{\xi}_{j'}) > 0\right\}.
\end{align*}
Crucially, note that the definitions of $Y_1$ and $Y_2$ depend exclusively on the samples in $[\bm{\xi}]_{\bm{Q}}$, and are therefore independent of the remaining samples $[\bm{\xi}]_{[N] \setminus \bm{Q}} = (\bm{\xi}_i)_{i \in [N] \setminus \bm{Q}}$.
Therefore, conditioning on fixed $R_i \subseteq \bm{Q}_i \subseteq [N]$, $[\bm{\xi}]_{\bm{Q}}$, and consequently fixed sets $Y_1, Y_2$, we can decompose the sum as follows:
\begin{align*}
\mathsf{W}_{[\bm{\xi}]}(R_{2}^{+},R_{1}^{+}) &= \sum_{i, j \in [N]} \mathsf{W}(\bm{\xi}_i, \bm{\xi}_j) \mathbbm{1}_{\bm{\xi}_{i} \in Y_2} \mathbbm{1}_{\bm{\xi}_{j} \in Y_1}
\\
&= \sum_{i, j \in [N] \setminus \bm{Q}} \mathsf{W}(\bm{\xi}_i, \bm{\xi}_j) \mathbbm{1}_{\bm{\xi}_{i} \in Y_2} \mathbbm{1}_{\bm{\xi}_{j} \in Y_1} + \sum_{i \in \bm{Q} \text{ or } j \in \bm{Q}} \mathsf{W}(\bm{\xi}_i, \bm{\xi}_j) \mathbbm{1}_{\bm{\xi}_{i} \in Y_2} \mathbbm{1}_{\bm{\xi}_{j} \in Y_1}.
\end{align*}
For the first term, we take the expectation over the remaining samples $[\bm{\xi}]_{[N] \setminus \bm{Q}}$ and obtain
\begin{align*}
&\E_{[\bm{\xi}]_{[N] \setminus \bm{Q}}} \bigg[ \sum_{i, j \in [N] \setminus \bm{Q}} \mathsf{W}(\bm{\xi}_i, \bm{\xi}_j) \mathbbm{1}_{\bm{\xi}_{i} \in Y_2} \mathbbm{1}_{\bm{\xi}_{j} \in Y_1} \bigg]
\\
& \quad = \sum_{i, j \in [N] \setminus \bm{Q}} \int_{\Omega \times \Omega} \mathsf{W}(\zeta, \eta) \mathbbm{1}_{\zeta \in Y_2} \mathbbm{1}_{\eta \in Y_1} \, \mathrm{d} \pi_i(\zeta) \, \mathrm{d} \pi_j(\eta)
\\
& \quad \leq N^2 \|\mathsf{W}\|_{\square}^{+}.
\end{align*}
The second term contains no more than $2N|\bm{Q}| \leq 4Nq$ entries. Hence,
\begin{align*}
\E_{[\bm{\xi}]_{[N] \setminus \bm{Q}}} \big[ \mathsf{W}_{[\bm{\xi}]}(R_{2}^{+},R_{1}^{+}) \big] \leq N^2 \|\mathsf{W}\|_{\square}^{+} + 4\|\mathsf{W}\|_{L^\infty}Nq.
\end{align*}

Then, following a similar sample concentration argument, we can obtain an $N$-certain bound. Note that changing one sample $\bm{\xi}_k$ (where $k \in [N] \setminus \bm{Q}$) alters the sum $\mathsf{W}_{[\bm{\xi}]}(R_{2}^{+},R_{1}^{+})$ by at most $4N\|\mathsf{W}\|_{L^\infty}$. By Azuma's inequality, with probability at least $1 - e^{-1.9 q}$,
\begin{align*}
\mathsf{W}_{[\bm{\xi}]}(R_{2}^{+},R_{1}^{+}) &\leq \E_{[\bm{\xi}]_{[N] \setminus \bm{Q}}} \big[ \mathsf{W}_{[\bm{\xi}]}(R_{2}^{+},R_{1}^{+}) \big] + 7.9 \|\mathsf{W}\|_{L^\infty} N\sqrt{N q}
\\
&\leq N^2 \|\mathsf{W}\|_{\square}^{+} + 4\|\mathsf{W}\|_{L^\infty}Nq + 7.9 \|\mathsf{W}\|_{L^\infty} N\sqrt{N q}.
\end{align*}
Recall that we take $q=\lfloor\sqrt{N}/4\rfloor$. The number of possible pairs of sets $R_1 \subseteq \bm{Q}_1$ and $R_2 \subseteq \bm{Q}_2$ is $4^q$, which grows slower than $e^{1.9 q}$. Hence, by taking the maximum over all $R_1$ and $R_2$, one can still obtain $N$-certain bounds:
\begin{align*}
\max_{\substack{R_{1}\subseteq \bm{Q}_{1}, R_{2}\subseteq \bm{Q}_{2}}} \mathsf{W}_{[\bm{\xi}]}(R_{2}^{+},R_{1}^{+}) \leq N^2 \|\mathsf{W}\|_{\square}^{+} + 4\|\mathsf{W}\|_{L^\infty}Nq + 7.9 \|\mathsf{W}\|_{L^\infty} N\sqrt{N q}.
\end{align*}
This bound holds for the random samples $[\bm{\xi}]_{[N]\setminus\bm{Q}}$, conditional on any fixed realization of $\bm{Q}$ and $[\bm{\xi}]_{\bm{Q}}$.

Substituting this back into the expectation over $\bm{Q}_{1}, \bm{Q}_{2}$ from Lemma~\ref{lem:upper_bound_one_sided_cut_norm}, we obtain:
\begin{align*}
\|\mathsf{W}_{[\bm{\xi}]}\|_{\square}^{+} &\leq \frac{1}{N^{2}}\E_{\bm{Q}_{1},\bm{Q}_{2}}\left[\max_{\substack{R_{1}\subseteq \bm{Q}_{1}, R_{2}\subseteq \bm{Q}_{2}}}  \mathsf{W}_{[\bm{\xi}]}(R_{2}^{+},R_{1}^{+})\right]+\frac{2 \|\mathsf{W}\|_{L^\infty}^{}}{\sqrt{q}}
\\
&\leq \|\mathsf{W}\|_{\square}^{+} + 4\|\mathsf{W}\|_{L^\infty} \frac{q}{N} + 7.9 \|\mathsf{W}\|_{L^\infty} \sqrt{\frac{q}{N}} + 2 \|\mathsf{W}\|_{L^\infty}^{} \frac{1}{\sqrt{q}}
\\
& \leq \|\mathsf{W}\|_{\square}^{+} + \frac{C \|\mathsf{W}\|_{L^\infty}}{N^{1/4}}.
\end{align*}
Finally, by applying the same argument for $\|-\mathsf{W}\|_{\square}^{+}$, we conclude the proof.

\end{proof}

\section*{Acknowledgments}

The author would like to thank Nicolas Fournier and Pierre-Emmanuel Jabin for the valuable comments on the presentation after the first draft of this article was written.

\bibliography{adapting_weights_v5}{}
\bibliographystyle{siam} 

\end{document}